\documentclass{article}

\usepackage{color}
 \catcode`@=11
\usepackage{color}
\usepackage{amsthm,amssymb,amsfonts}
\usepackage{amsmath}
\catcode`@=11 \@addtoreset{equation}{section}

\catcode`@=12

\newtheorem{theorem}{Theorem}[section]
\newtheorem{proposition}{Proposition}[section]
\newtheorem{lemma}{Lemma}[section]

\theoremstyle{definition}

\newtheorem{definition}{Definition}[section]
\newtheorem{remark}{Remark}

\newtheorem{hypothesis}{Hypothesis}[section]
\def\R{{\mathbb{R}}}

\newcommand{\ds}{\displaystyle}
\def\supp{\mathop{\rm supp}\nolimits}
\newcommand{\norm}[1]{\left\Vert#1\right\Vert}
\usepackage[colorlinks=true,urlcolor=red,linkcolor=blue,citecolor=red,bookmarks=red]{hyperref}

\title{Null controllability for degenerate parabolic equations with a nonlocal space term}
\author{{\sc Brahim Allal}
\\
Facult\'e des Sciences et Techniques\\
Universit\'e Hassan 1er\\
Laboratoire MISI, B.P. 577\\
Settat 26000, Morocco
\\ email: b.allal@uhp.ac.ma
\\
{\sc Genni Fragnelli}
\thanks{The author is a member of the Gruppo Nazionale per l'Analisi Ma\-te\-matica, la Probabilit\`a e le loro Applicazioni (GNAMPA) of the Istituto Nazionale di Alta Matematica (INdAM) and she is supported by the FFABR {\it Fondo per il finanziamento delle attivit\`a base di ricerca} 2017, by the INdAM - GNAMPA Project 2020 {\it Problemi inversi e di controllo per equazioni di evoluzione e loro applicazioni},  by  Fondi di Ateneo 2017/18 of the University of  Bari {\em Problemi differenziali non linearii} and by PRIN 2017-2019 {\it Qualitative and quantitative aspects of nonlinear PDEs}.}
\\
Dipartimento di Matematica\\ Universit\`{a} di Bari Aldo Moro\\
Via
E. Orabona 4\\ 70125 Bari - Italy\\ email: genni.fragnelli@uniba.it\\
{\sc  Jawad Salhi}\\
Moulay Ismail University of Meknes,\\
FST Errachidia, MAIS Laboratory, MAMCS Group,\\
P.O. Box 509, Boutalamine 52000, Errachidia, Morocco\\
email: sj.salhi@gmail.com}

\date{}
\begin{document}

\maketitle
\begin{abstract}
We consider two degenerate heat equations with a nonlocal space term, studying, in particular, their null controllability property.
To this aim, we first consider the associated nonhomogeneous degenerate heat equations: we study their well posedness, the Carleman estimates for the associated adjoint problems and, finally, the null controllability. Then, as a consequence, using the Kakutani's fixed point Theorem,  we deduce the null controllability property for the initial nonlocal problems.
\end{abstract}
\noindent {\bf{Keywords:}} Controllability, degenerate equationsl, nonlocal term, Carleman estimates.
\smallskip

% Please enter at most 5 Mathematics Subject Classification codes here. Please use 2010 classification codes, which can be found on the following link: http://www.ams.org/msc//msc2010.html.
\noindent{\bf{2020 Mathematics Subject Classification:}} 93B05, 35K05, 35K67, 35R09.

\section{Introduction}
\label{intro}

In this paper, we address the null controllability for the following degenerate integro-differential equations in  non divergence and divergence form:
\begin{equation}\label{problem_nondiv}
\begin{cases}
\ds y_t - ay_{xx} + \int_{0}^{1} K(t,x,\tau)y(t,\tau)\,d\tau  = 1_{\omega} u, & (t,x) \in Q,\\
y (t, 0) = y (t, 1) =0,& t \in (0,T),
\\
y(0,x)=y_0(x), & x \in (0,1)
\end{cases}
\end{equation}
and
\begin{equation}\label{problem_div}
\left\{
  \begin{array}{ll}
y_t - (ay_{x})_x + \int_{0}^{1} K(t,x,\tau)y(t,\tau)\,d\tau =  1_{\omega} u, & (t,x) \in Q,\\
y(t, 1) = 0, & t \in (0,T),\\
\left\{
\begin{array}{ll}
y (t, 0) = 0,\quad \text{in the weakly degenerate case (WD)},\\
(a y_{x})(t, 0)= 0, \quad \text{in the strongly degenerate case (SD)},
\end{array}
\right.
& t \in (0,T),
\\
y(0,x)=y_0(x), & x \in (0,1),
  \end{array}
\right.
\end{equation}
\noindent where $T>0$ is fixed, $Q:=(0,T) \times (0,1)$ and $1_{\omega}$ denotes the characteristic function of a nonempty open subset $\omega:=(\bar{\alpha},\bar{\beta})$ compactly contained in $(0,1)$. Here $y$ and $u$ are the state variable and the control force respectively, $K$ is a given $L^{\infty}$ function defined on $Q\times(0,1)$ and $a$ is a diffusion coefficient which degenerates at the extremity $x=0$ and satisfies the following hypotheses: \\

\noindent In the case of nondivergence form:
\begin{hypothesis}\label{Hypoth_1_nondiv}
The function $a\in C[0,1]$ is such that $a(0)=0, \, a>0 \; \text{in}\; (0, 1]$ and there exists $\alpha \in (0,2),$ such that the function $ x \mapsto \frac{x^{\alpha}}{a(x)}$ is nondecreasing.
\end{hypothesis}
\noindent In the case of divergence form:
\begin{hypothesis}\label{Hypoth_WD_div} {\bf Weakly degenerate (WD)}
The function $ a \in C([0,1]) \cap C^1((0,1])$ is such that  $a(0)=0$, $a>0$ in $(0, 1]$  and $\frac{1}{a} \in L^1(0,1)$.
\end{hypothesis}
\begin{hypothesis}\label{Hypoth_SD_div} {\bf Strongly degenerate (SD)}
The function $ a \in C^1([0,1])$ is such that  $a(0)=0$, $a>0$ in $(0, 1]$  and $\frac{1}{\sqrt a} \in L^1(0,1)$.
\end{hypothesis}

\begin{remark}\label{remark_WSD}
Thanks to Hypothesis \ref{Hypoth_1_nondiv} we see that  $ x \mapsto \frac{x^{\alpha}}{a(x)}$ is nondecreasing on $(0, 1]$ and thus
$$ \frac{1}{a(x)} \leq  \frac{1}{x^{\alpha}a(1)}.$$
This implies that $\frac{1}{a} \in L^1(0,1)$ if $\alpha \in (0,1)$ and $\frac{1}{\sqrt a} \in L^1(0,1)$ if $\alpha \in [1,2)$.
\end{remark}

We emphasize the fact that, in this work, problems \eqref{problem_nondiv} and \eqref{problem_div} will be treated separately, since the controllability property of the first one cannot be deduced from the one in the divergence form. Indeed, for instance, in the absence of the nonlocal term, the equation
$$ \displaystyle y_t - ay_{xx} = 0 $$
can be rewritten as
\begin{equation}\label{equation_div_firstordr}
\displaystyle y_t - (ay_x)_x + a_xy_{x} = 0
\end{equation}
only if $a_x$ exists.
Moreover, as described in
\cite{bfr},
degenerate equations of the form \eqref{equation_div_firstordr},
are well-posed in $L^2(0,1)$ under the structural assumption
\begin{equation} \label{pri}
|a_x(x)|\le C\sqrt{a(x)}
\end{equation}
where $C$  is a positive constant.
Now, imposing \eqref{pri} on $a_x$, for $a(x)=x^\alpha$, we obtain $\alpha\ge 2$; nevertheless, in \cite{CFR2008}, it is proved that \eqref{equation_div_firstordr} is not null controllable if $\alpha\geq 2.$

Null controllability for nonlocal parabolic problems of type  \eqref{problem_nondiv} or \eqref{problem_div} has recently attracted the attention of many mathematicians, since these problems  describe a variety of physical phenomena. For instance, when $a=1$, such system appears in population dynamics, where $y(t,x)$ represents the density of the species at position $x$ and time $t$, while the reaction term $\int_{0}^{1} K(\cdot,\cdot,\tau)y(\cdot,\tau)\,d\tau$ is considered as the rate of reproduction. This integral term is a way to express that the evolution of the species in a point of space depends on the total amount of the species (see for instance \cite{Calsina, Furter, Zheng}).

In \cite{CLZ2016}, in the context of uniformly parabolic equations and using an  approach  based on a compactness-uniqueness argument, the authors have established the controllability for \eqref{problem_nondiv} assuming that the kernel $K$ is a time-independent and analytic function; in the same paper  a similar result is also shown for the wave equation. The result of \cite{CLZ2016} is extended later in \cite{Lissy2018} to a general coupled parabolic system and in \cite{Micu} to a  $1-d$ scalar equation, by assuming a particular class of time-independent kernels in separated variables. More precisely, in \cite{Micu}, the kernel $K$ is such that $ K(t,x,y)=K_1(x)K_2(y),$ with $K_1$ not vanishing in the region where the control acts. Recently, in \cite{Biccari2019} U. Biccari and V. H. Santamar\'{\i}a have extended the last result considering a problem in any space dimension and relaxing the assumptions on the kernel. In particular, in \cite{Biccari2019} it is shown that the considered system is null controllable provided the function $K$ enjoys only an exponential decay at the final time $t = T$.

Finally, we would like to mention \cite{Lorenzi11},  where unique continuation and inverse problem of an integro-differential equation are analyzed via Carleman estimates. In particular, in \cite{Lorenzi11} it is considered an integral term involving the solution and its first order derivatives.

The main goal of this paper is to provide a suitable condition on the kernel $K$ so that the degenerate problem \eqref{problem_nondiv} (resp. \eqref{problem_div}) is null controllable, that is to say, for any initial data $y_0$, there exists a control function $u$ such that the associated solution to \eqref{problem_nondiv} (resp. \eqref{problem_div}) vanishes at a given time $t=T$.

\begin{remark}
It is well known that, the null controllability of system \eqref{problem_nondiv} is equivalent to the proof of the observability inequality
\begin{equation}\label{observ_nondiv}
\|v(0)\|_{L^2(0,1)}^2 \leq C \int_0^T\!\!\! \int_{\omega} v^2 \,dx\,dt,\qquad \forall v_T \in L^2(0,1)
\end{equation}
for the solutions of the adjoint system
\begin{equation}\label{problem_adj_nondiv}
\left\{
\begin{array}{ll}
\displaystyle - v_t - av_{xx} + \int_{0}^{1} K(t,\tau,x)v(t,\tau)\,d\tau  = 0, & (t,x) \in Q,\\
v (t, 0) = v (t, 1) =0,& t \in (0,T),
\\
v(T,x)=v_T(x),&  x\in (0,1).
\end{array}
\right.
\end{equation}
Recall that the classical way to establish an estimate of this kind is to derive a global Carleman inequality of the form
\begin{align}\label{Carleman_remark1}
\int\!\!\!\!\!\int_{Q} \Big(s \theta   v_{x}^{2} + s^{3} \theta^3 \left(\frac{x}{a}\right)^2 v^{2}\Big) e^{2s\phi}\,dx\,dt
\leq C \int_0^T\!\!\! \int_{\omega} v^2 \,dx dt,
\end{align}
for all $s$ large enough. Here $\theta$ and $\phi$ are as in \eqref{phi}-\eqref{weightfunc_deg_nodiv}. In fact, once the previous estimate holds, it suffices to apply the Hardy-Poincar\'e inequality \eqref{hardyineq_nondiv} and to use a standard calculation (which is based on integration by parts) to obtain \eqref{observ_nondiv}. Nevertheless, it is not clear whether an estimate of this type can be achieved in the context of nonlocal problems.

Indeed, an application of the Carleman estimate obtained in Theorem \ref{thm_Carl_local} to \eqref{problem_adj_nondiv} yields
\begin{align}\label{Carleman_remark2}
&\int\!\!\!\!\!\int_{Q} \Big(s \theta   v_{x}^{2} + s^{3} \theta^3 \left(\frac{x}{a}\right)^2 v^{2}\Big) e^{2s\phi}\,dx\,dt \notag \\
&\leq C \int\!\!\!\!\!\int_{Q} \frac{e^{2s\Phi}}{a}\Big(\int_{0}^{1} K(t,\tau,x)v(t,\tau)\,d\tau\Big)^2 \,dx\,dt  + C\int_0^T\!\!\! \int_{\omega} v^2 \,dx dt,
\end{align}
where $\Phi(t,x):=\theta(t)\Psi(x)$ is an appropriate weight function satisfying
$\phi\leq\Phi<0$ (see \eqref{weightfunc_nondeg}). Then, in order to deduce \eqref{Carleman_remark1}, it suffices to show that the first term in the right hand side of \eqref{Carleman_remark2} can be controlled by
the left hand side; but this appears to be impossible due to the nonlocal nature of this term. Hence, the usual technique do not seem to work directly for the controllability problems of integro-differential parabolic equations like \eqref{problem_nondiv}.

The same can be said about system \eqref{problem_div}.
\end{remark}

To overcome this difficulty and establish the desired controllability results for nonlocal problems \eqref{problem_nondiv} and  \eqref{problem_div}, we use the following arguments:
\begin{itemize}
\item Step 1: We establish the null controllability property for appropriate nonhomogeneous systems without nonlocal term, via new Carleman estimates with weight time functions that do no blow up at $t=0$.
\item Step 2: As consequence of the result in the previous step, we arrive at the controllability for the nonlocal problems by means of a fixed point argument.
\end{itemize}

The rest of the paper is organized as follows. In the next section, using the Faedo-Galerkin method, the global existence and uniqueness
of a weak solution to problem \eqref{problem_nondiv} is proved. Section \eqref{section_null_nonhomo} is devoted to the analysis of the
null controllability property for the two associated nonhomogeneous degenerate problems, without nonlocal term, in a suitable functional setting.
Section \eqref{section_null_nonlocal_nondiv} is concerned with the null controllability for the initial nonlocal problem \eqref{problem_nondiv} and in Section \eqref{section_null_nonlocal_div} we give a sketch of the proof for the null controllability of \eqref{problem_div}. Finally, in Section \eqref{Appendix}, we give the proof of some technical results.

A final comment on the notation: by $C$ we shall denote universal positive constants,
which are allowed to vary from line to line.

\section{Well-posedness}
In this section, we prove the well-posedness (existence and uniqueness) of the weak solution to problems \eqref{problem_nondiv} and \eqref{problem_div}.
To prove well posedness of \eqref{problem_nondiv}, as in \cite{CFR2008}, we assume that
\[
a\in C[0,1] \text{ is such that } a(0)=0, \, a>0 \; \text{in}\; (0, 1].
\]
and we introduce the following weighted Hilbert spaces
$$ L_{\frac{1}{a}}^2(0,1):= \Big\{ u \in L^2(0,1)\, | \quad \frac{u}{\sqrt{a}} \in L^2(0,1) \Big\}, $$
$$H_{\frac{1}{a}}^1(0,1) := L_{\frac{1}{a}}^2(0,1) \cap H_0^1(0,1)  $$
and
$$ H_{\frac{1}{a}}^2(0,1) :=  \Big\{ u \in   H_{\frac{1}{a}}^1(0,1) \, | \quad  au_{xx} \in   L_{\frac{1}{a}}^2(0,1) \Big\}$$
endowed with the following norms
\begin{align*}
&\| u \|_{L_{\frac{1}{a}}^2(0,1)}^2 : = \int_0^1 \frac{u^2}{a} \, dx, \quad \forall \;u \in L_{\frac{1}{a}}^2(0,1),\\
&\| u \|_{H_{\frac{1}{a}}^1(0,1)}^2:=  \| u \|_{ L_{\frac{1}{a}}^2(0,1)}^2 +  \| u_x \|_{L^2(0,1)}^2, \quad \forall\; u \in   H_{\frac{1}{a}}^1(0,1) , \\
&\| u \|_{H_{\frac{1}{a}}^2(0,1)}^2 :=  \| u \|_{H_{\frac{1}{a}}^1(0,1)}^2 +  \| a u_{xx} \|_{L^2_{\frac{1}{a}}(0,1)}^2, \quad \forall\; u \in  H_{\frac{1}{a}}^2(0,1).
\end{align*}
Let $H_{\frac{1}{a}}^{-1}(0,1)$ be the dual space of $H_{\frac{1}{a}}^1(0,1)$ with respect to the pivot space
$L_{\frac{1}{a}}^2(0,1)$, endowed with the natural norm
\begin{equation*}
\| z \|_{H_{\frac{1}{a}}^{-1}}:=  \sup_{\| y \|_{H_{\frac{1}{a}}^{1}}=1} \langle z, y \rangle_{H_{\frac{1}{a}}^{-1}, H_{\frac{1}{a}}^{1}}.
\end{equation*}
While in the divergence case, as in \cite{Alabau2006}, we consider the following weighted spaces.
In the (WD) case:
\begin{align*}
H_a^1(0,1):= \Big\{ y \in L^2(0,1): y\; &\text{a.c. in}\, [0,1],\; \sqrt{a}y_x \in L^2(0,1)\;\text{and}\; y(1)=y(0)=0 \Big\}
\end{align*}
and
\begin{align*}
H_a^2(0,1):= \Big\{ y \in H_a^1(0, 1): ay_x \in H^1(0,1)\Big\}.
\end{align*}
In the (SD) case:
\begin{align*}
H_a^1(0,1):= \Big\{ y \in L^2(0,1): y\, &\text{locally a.c. in}\, (0,1],\quad\sqrt{a}y_x \in L^2(0,1)\,\text{and}\, y(1)=0 \Big\}
\end{align*}
and
\begin{align*}
H_a^2(0,1):&= \Big\{ y \in H_a^1(0, 1): ay_x \in H^1(0,1)\Big\}\\
&=\Big\{ y \in L^2(0,1): y\, \text{locally a.c. in}\, (0,1], ay\in H^1_0(0,1),\\
&\qquad ay_x \in H^1(0,1) \,\text{and}\, (ay_x)(0)=0 \Big\}.
\end{align*}
In both cases, we consider the following norms
\begin{align*}
&\| y \|_{H_a^1(0,1)}^2:=  \| y \|_{L^2(0,1)}^2 +  \| \sqrt{a} y_x \|_{L^2(0,1)}^2\qquad\text{and}\\
&\| y \|_{H_a^2(0,1)}^2 :=  \| y \|_{H_a^1(0,1)}^2 +  \| (a y_x )_x \|_{L^2(0,1)}^2.
\end{align*}

%%%%%%%%%%%%%%%%%%%%%%%%%%%%%%%%
Let us give the definition of weak solutions to problems \eqref{problem_nondiv} and \eqref{problem_div}.
\begin{definition}
Let $y_0\in L_{\frac{1}{a}}^2(0,1)$ and $u\in  L_{\frac{1}{a}}^2(Q):= L^2(0,T;L_{\frac{1}{a}}^2(0,1))$. A function $y$ is said to be a weak solution of problem \eqref{problem_nondiv}, if
\begin{equation*}
\begin{cases}
&y \in L^2(0, T; H_{\frac{1}{a}}^1(0,1)), \quad y_t \in L^2(0, T; H_{\frac{1}{a}}^{-1}(0,1)), \\
&\int\!\!\!\int_{Q} (\frac{y_t v}{a} + y_x v_x ) \, dt dx \notag\\
&= - \int\!\!\!\int_{Q} \left(\int_{0}^{1} K(t,x,\tau)y(t,\tau)\,d\tau \right) \frac{v}{a} \, dt dx + \int\!\!\!\int_{Q_{\omega}} \frac{u v}{a} \, dt dx,\quad \forall v \in L^2(0, T; H_{\frac{1}{a}}^1(0,1))\quad\text{and}\\
&y(0)=y_0.
\end{cases}
\end{equation*}
Here $Q_{\omega}:= (0,T) \times \omega$.
\end{definition}

\begin{definition}
Let $y_0\in L^2(0,1)$ and $u\in  L^2(Q)$. A function $y$ is said to be a weak solution of problem \eqref{problem_div}, if
\begin{equation*}
\begin{cases}
&y \in L^2(0, T; H_{a}^1(0,1)), \quad y_t \in L^2(0, T; H_{a}^{-1}(0,1)), \\
&\int\!\!\!\int_{Q} (y_t v + a y_x v_x ) \, dt dx \notag\\
&= - \int\!\!\!\int_{Q} \left(\int_{0}^{1} K(t,x,\tau)y(t,\tau)\,d\tau \right) v\, dt dx + \int\!\!\!\int_{Q_{\omega}} u v \, dt dx,
\quad \forall v \in L^2(0, T; H_{a}^1(0,1))\quad \text{and} \\
&y(0)=y_0.
\end{cases}
\end{equation*}
Here $H_{a}^{-1}(0,1)$ denotes the dual of $H_{a}^1(0,1)$.
\end{definition}

Now we are ready for the main results of this Section.
\begin{theorem}\label{Thm_wellposed_nonloc_nondiv}
Let $y_0 \in L_{\frac{1}{a}}^2(0,1)$, $u\in  L_{\frac{1}{a}}^2(Q)$ and assume that the kernel $K$ satisfies the following condition:
\begin{equation}\label{hyp_K_wellposed}
\int_{0}^{1}\int_{0}^{1}  \frac{K^2(t,x,\tau)}{a(x)}\,d\tau\,dx \in L^{\infty}(0,T).
\end{equation}
Then, system \eqref{problem_nondiv} admits a unique weak solution $y$ such that
$$
y \in C([0,T];  L^2_{\frac{1}{a}}(0,1))\cap L^2(0, T; H_{\frac{1}{a}}^1(0,1)), \quad y_t \in L^2(0, T; H_{\frac{1}{a}}^{-1}(0,1)).
$$
Furthermore, there is a positive constant $C$ such that
\[\begin{aligned}
& \sup\limits_{t\in[0,T]}\|y(t)\|_{L^2_{\frac{1}{a}}(0,1)}^2 + \int_{0}^{T} \|y(t)\|_{H_{\frac{1}{a}}^1(0,1)}^2\,dt \\
&+  \int_{0}^{T} \|y_{t}(t)\|_{H_{\frac{1}{a}}^{-1}(0,1)}^2\,dt\leq C \left( \|y_0\|_{L^2_{\frac{1}{a}}(0,1)}^2 + \|u\|_{L^2_{\frac{1}{a}}(Q_{\omega})}^2 \right).
\end{aligned}\]
\end{theorem}

\begin{theorem}\label{Thm_wellposed_nonloc_div}
Let $y_0 \in L^2(0,1)$ and $u\in  L^2(Q)$. Then, system \eqref{problem_div} admits a unique weak solution $y$ such that
$$
y \in C([0,T];  L^2(0,1))\cap L^2(0, T; H_{a}^1(0,1)), \quad y_t \in L^2(0, T; H_{a}^{-1}(0,1)).
$$
Furthermore, there is a positive constant $C$ such that
\[\begin{aligned}
& \sup\limits_{t\in[0,T]}\|y(t)\|_{L^2(0,1)}^2 + \int_{0}^{T} \|y(t)\|_{H_{a}^1(0,1)}^2\,dt \\
&+  \int_{0}^{T} \|y_{t}(t)\|_{H_{a}^{-1}(0,1)}^2\,dt\leq C \left( \|y_0\|_{L^2(0,1)}^2 + \|u\|_{L^2(Q_{\omega})}^2 \right).
\end{aligned}
\]
\end{theorem}

\begin{remark}
\begin{itemize}
  \item Notice that, since $K \in L^{\infty}(Q\times (0,1))$, the assumption \eqref{hyp_K_wellposed} holds immediately when $\frac{1}{a} \in L^1(0,1)$. On the other hand, if $\frac{1}{\sqrt a} \in L^1(0,1)$, the function $K$ needs to satisfy
$$ |K(t,x,\tau)| \leq C a^{\frac{1}{4}}(x), $$
for some positive constant $C$.
  \item Contrary to the non divergence case, we point out that we do not need to impose any additional condition on the kernel $K\in L^{\infty}(Q\times(0,1))$ to ensure the well-posedness of problem \eqref{problem_div}.
\end{itemize}
\end{remark}

In what follows we only give the detailed proof for the non divergence case. For the divergence case, the proof is essentially the same, and we omit the detail here.
%%%%%%%%%%%%%%%%%%%%%%%%%%%%%%%%%%%%%%%%%%%%%%%%%%%%%%%%%%%%%%%%
\begin{proof}[Proof of Theorem \ref{Thm_wellposed_nonloc_nondiv}]
We will use the well-known Faedo-Galerkin method.\\
\noindent\textbf{step 1.}(Faedo-Galerkin approximation)

Let $\{w_k\}_{k\geq1}$ be an orthogonal basis of the Hilbert spaces $H_{\frac{1}{a}}^1(0,1)$ and $L_{\frac{1}{a}}^2(0, 1)$. We assume, by normalization of $w_k$ in $L_{\frac{1}{a}}^2(0,1)$, that $\| w_k\|_{L_{\frac{1}{a}}^2(0,1)} = 1, \quad \forall\; k \geq 1$.

For each integer $m\geq 1$, we consider $ V_m = [w_1, w_2, \dots, w_m]$, the subspace generated by the
first $m$ vectors of $\{w_k\}_{k\geq1}$. Let us also define the orthogonal projection
$$
\mathcal{P}_m: L^2_{\frac{1}{a}}(0,1) \rightarrow V_m \subset L^2_{\frac{1}{a}}(0,1),
$$
as
$$
\int_{0}^{1}\frac{\mathcal{P}_m(y)v}{a}\,dx= \int_{0}^{1}\frac{y v}{a}\,dx,\quad \text{for all}\quad y\in L^2_{\frac{1}{a}}(0,1)\quad \text{and}\quad v\in V_m.
$$
Let us set $h=1_{\omega} u \in L_{\frac{1}{a}}^2(Q)$. We are looking for an approximate solution $y_m$ for \eqref{problem_nondiv} under the form
\begin{equation}\label{ym}
y_m(t,x) = \sum_{k=1}^{m} \alpha_k^m(t) w_k(x),
\end{equation}
so that
\begin{equation}\label{Pm}
\begin{aligned}
\int_{0}^{1}   \frac{y_{m,t} w_k}{a} \, dx + \int_{0}^{1}  y_{m,x} w_{k,x} \, dx &= - \int_{0}^{1} \left(\int_{0}^{1} K(t,x,\tau)y_m(t,\tau)\,d\tau \right) \frac{w_k}{a} \, dx \\
&\quad+ \int_{0}^{1}  \frac{h w_k}{a} \,dx
\end{aligned}
\end{equation}
for $k=1, \dots, m$, where the initial conditions are such that
\begin{equation}\label{Pmid}
\int_{0}^{1}\frac{y_m(0,x)w_k(x)}{a}\,dx= \int_{0}^{1}\frac{y_0(x)w_k(x)}{a}\,dx.
\end{equation}
The function $y_m(t,x) = \sum_{k=1}^{m} \alpha_k^m(t) w_k(x)$ is a solution of \eqref{Pm}-\eqref{Pmid} if $\alpha^m=(\alpha_1^m,\dots,\alpha_m^m)$ is a solution of the system of ordinary differential equations
\begin{equation}\label{ODEm}
\begin{cases}
&\displaystyle \frac{d}{dt} \alpha_k^m(t) + \Lambda_k \alpha_k^m(t) = F( \alpha_1^m(t), \alpha_2^m(t), \dots,  \alpha_m^m(t))\\
&\alpha_k^m(0) = (y_0, w_k),\\
& k=1, \dots, m,
\end{cases}
\end{equation}
where
$$F(\alpha_1^m(t), \alpha_2^m(t), \dots,  \alpha_m^m(t)) := - \int_{0}^{1} \left(\int_{0}^{1} K(t,x,\tau)y_m(t,\tau)\,d\tau \right) \frac{w_k}{a} \,dx
+\int_{0}^{1} \frac{h w_k}{a}\,dx,$$
$\Lambda_k:= \int_{0}^{1} |w_{kx}|^2 \, dx$ and $(.,.)$ denotes the inner product in $L^2_{\frac{1}{a}}(0,1)$.

According to the classical theory of ordinary differential equations, the initial problem
\eqref{ODEm} admits local solutions, which further implies the local existence of solutions for the problem \eqref{Pm}-\eqref{Pmid}.

In the next step, we make some useful a priori estimates to make sure that the function $\alpha^m$ is defined on the interval $(0, T )$ for all
$T > 0$, which enables us to get global solutions to the problem \eqref{Pm}-\eqref{Pmid}.

\noindent\textbf{step 2.}(A priori estimates)

We multiply \eqref{Pm} by $\alpha_k^m$ and sum up the resulting equations for $k= 1,\cdots, m$. Then, we can easily get
\begin{align}\label{A1}
&\frac{1}{2}\frac{d}{dt} \int_{0}^{1} \frac{|y_m|^2 }{a} \, dx + \int_{0}^{1} |y_{m,x}|^2  \, dx \notag\\
&= - \int_{0}^{1} \left(\int_{0}^{1} K(t,x,\tau)y_m(t,\tau)\,d\tau \right) \frac{y_m}{a(x)} \, dx + \int_{0}^{1}  \frac{h y_m}{a} \, dx.
\end{align}
Applying the Cauchy-Schwarz inequality, we have that  for every $t \in (0, T]$,
\begin{align}\label{A2a}
&- \int_{0}^{1} \left(\int_{0}^{1} K(t,x,\tau)y_m(t,\tau)\,d\tau \right) \frac{y_m}{a(x)} \, dx\notag\\
&=-\int_{0}^{1} \left(\int_{0}^{1}\sqrt{a(\tau)} \frac{K(t,x,\tau)}{\sqrt{a(x)}} \frac{y_m(t,\tau)}{\sqrt{a(\tau)}}\,d\tau \right) \frac{y_m(t,x)}{\sqrt{a(x)}} \, dx \notag\\
& \leq \max\limits_{\tau\in[0,1]}\sqrt{a(\tau)} \Big(\int_{0}^{1}\int_{0}^{1}  \frac{K^2(t,x,\tau)}{a(x)}\,d\tau\,dx\Big)^{\frac{1}{2}} \Big(\int_0^1\frac{|y_m(t,\tau)|^2}{a(\tau)} d\tau\Big)^{\frac{1}{2}} \Big(\int_0^1\frac{|y_m(t,x)|^2}{a(x)} dx\Big)^{\frac{1}{2}} \notag\\
&= \max\limits_{\tau\in[0,1]}\sqrt{a(\tau)} \Big(\int_{0}^{1}\int_{0}^{1}  \frac{K^2(t,x,\tau)}{a(x)}\,d\tau\,dx\Big)^{\frac{1}{2}} \Big(\int_0^1\frac{|y_m(t,\tau)|^2}{a(\tau)} d\tau\Big)\notag\\
&\leq C \|y_m(t)\|_{L^2_{\frac{1}{a}}(0,1)}^2,
\end{align}
where
$C= \max\limits_{\tau\in[0,1]}\sqrt{a(\tau)}\sup\limits_{t\in (0,T)} (\int_{0}^{1}\int_{0}^{1}  \frac{K^2(t,x,\tau)}{a(x)}\,d\tau\,dx) ^{\frac{1}{2}}$.

By \eqref{A1}, \eqref{A2a} and Young's inequality, we obtain that there exists a constant $C > 0$
such that
\begin{align}\label{A2}
\frac{1}{2}\frac{d}{dt} \|y_m(t)\|_{L^2_{\frac{1}{a}}(0,1)}^2 & + \int_{0}^{1} |y_{m,x}|^2  \, dx \notag\\
& \leq C \big( \|y_m(t)\|_{L^2_{\frac{1}{a}}(0,1)}^2 + \|h(t)\|_{L^2_{\frac{1}{a}}(0,1)}^2 \big).
\end{align}
Employing Gronwall's inequality, we get that
\begin{align}\label{A3}
& \|y_m(t)\|_{L^2_{\frac{1}{a}}(0,1)}^2 \leq C \left( \|y_0\|_{L^2_{\frac{1}{a}}(0,1)}^2 + \|h\|_{L^2_{\frac{1}{a}}(Q)}^2 \right),
\end{align}
for every $t\leq T$. This gives
\begin{align}\label{A4}
& \sup\limits_{t\in[0,T]}\|y_m(t)\|_{L^2_{\frac{1}{a}}(0,1)}^2 \leq C \left( \|y_0\|_{L^2_{\frac{1}{a}}(0,1)}^2 + \|h\|_{L^2_{\frac{1}{a}}(Q)}^2 \right),
\end{align}
where $C$ is a positive constant depending on $a$, $K$ and $T$. It turns out from \eqref{A4} that $\alpha^m$ is bounded and therefore the solutions
$y_m$ to the initial value problem \eqref{Pm}-\eqref{Pmid} are global.

\noindent\textbf{step 3.}(Passage to the limit)

In this step, we will select from the above solutions $y_m$ of the approximate problems \eqref{Pm}-\eqref{Pmid} a subsequence
and prove that when letting $m \rightarrow \infty$, they converge to a global weak solution for \eqref{problem_nondiv}. To this aim, we will need some uniform estimates.

We claim that there exists a positive constant $C$ which is independent of $m$, such that
\begin{align}\label{A7}
& \sup\limits_{t\in[0,T]}\|y_m(t)\|_{L^2_{\frac{1}{a}}(0,1)}^2 + \int_{0}^{T} \|y_m(t)\|_{H_{\frac{1}{a}}^1(0,1)}^2\,dt \notag\\
&+  \int_{0}^{T} \|y_{m,t}(t)\|_{H_{\frac{1}{a}}^{-1}(0,1)}^2\,dt\leq C \left( \|y_0\|_{L^2_{\frac{1}{a}}(0,1)}^2 + \|h\|_{L^2_{\frac{1}{a}}(Q)}^2 \right).
\end{align}
Indeed, from estimates \eqref{A2} and \eqref{A3} we immediately obtain
\begin{align*}
\int_{0}^{T}\!\!\!\int_{0}^{1} |y_{m,x}|^2\,dxdt\leq C \left(\|y_m(0)\|_{L^2_{\frac{1}{a}}(0,1)}^2 + \|h\|_{L^2_{\frac{1}{a}}(Q)}^2 \right),
\end{align*}
which, together with \eqref{A3}, implies that
\begin{align}\label{A5}
\int_{0}^{T} \|y_m(t)\|_{H_{\frac{1}{a}}^1(0,1)}^2\,dt \leq C \left( \|y_0\|_{L^2_{\frac{1}{a}}(0,1)}^2 + \|h\|_{L^2_{\frac{1}{a}}(Q)}^2 \right).
\end{align}
Combining \eqref{A4} and \eqref{A5}, we deduce that
\begin{align}\label{A6}
& \sup\limits_{t\in[0,T]}\|y_m(t)\|_{L_{\frac{1}{a}}^2(0,1)}^2 + \int_{0}^{T} \|y_m(t)\|_{H_{\frac{1}{a}}^1(0,1)}^2\,dt \notag\\
&\leq C \left( \|y_0\|_{L_{\frac{1}{a}}^2(0,1)}^2 + \|h\|_{L_{\frac{1}{a}}^2(Q)}^2 \right).
\end{align}
On the other hand, reasoning as in \cite[Theorem 2, Chapter 7]{Evans}, one can show that
\begin{align*}
\int_{0}^{T} \|y_{m,t}(t)\|_{H_{\frac{1}{a}}^{-1}(0,1)}^2\,dt
\leq C \left( \|y_0\|_{L^2_{\frac{1}{a}}(0,1)}^2 + \|h\|_{L^2_{\frac{1}{a}}(Q)}^2 \right).
\end{align*}
The last inequality together with \eqref{A6} gives \eqref{A7}.

Next, in order to build a weak solution of our initial/boundary-value problem \eqref{problem_nondiv}, we pass to limits as $m\rightarrow \infty$.
In fact, observe that from the energy estimates \eqref{A7}, we have that $\{y_m\}_{m\geq1}$ and $\{y_{m,t}\}_{m\geq1}$ are bounded in $L^2(0, T; H_{\frac{1}{a}}^1(0,1))$ and $L^2(0, T; H_{\frac{1}{a}}^{-1}(0,1))$. Hence, there exists a subsequence $\{y_{m_l}\}_{l\geq1}$ and a function $y \in L^2(0, T; H_{\frac{1}{a}}^1(0,1))$, with $y_t \in L^2(0, T; H_{\frac{1}{a}}^{-1}(0,1))$, such that
\begin{equation}\label{cv_seq}
\left\{
\begin{array}{ll}
y_{m_l} \rightharpoonup y \quad \text{weakly in} \quad L^2(0, T; H_{\frac{1}{a}}^1(0,1)) \\
y_{m_l,t} \rightharpoonup y_t \quad \text{weakly in} \quad L^2(0, T; H_{\frac{1}{a}}^{-1}(0,1)).
\end{array}
\right.
\end{equation}

Now, fix an integer $M \in \mathbb{N}$ and take a function $v\in C^1([0,T]; H_{\frac{1}{a}}^1(0,1))$ having the form
\begin{equation}\label{fucnt_v_p}
v(t) = \sum_{k=1}^{M} \alpha_k(t) w_k,
\end{equation}
where $\alpha_k$ $(1\leq k\leq M)$ are given smooth functions.

Let $m > M$. Multiplying \eqref{Pm} by $\alpha_k(t)$, summing up
with respect to $k$ and integrating on $(0, T)$, we derive
\begin{align}\label{step3_a}
& \int\!\!\!\!\!\int_{Q} \frac{y_{m,t} v}{a} \, dt + \int\!\!\!\!\!\int_{Q} y_{m,x} v_{x} \, dx dt\notag\\
& = - \int\!\!\!\!\!\int_{Q} \left(\int_{0}^{1} K(t,x,\tau)y_m(t,\tau)\,d\tau \right) \frac{v}{a} \, dx dt+ \int\!\!\!\!\!\int_{Q}  \frac{h v}{a} \, dx dt.
\end{align}
Set $m=m_l$. Passing to the limit we immediately get
\begin{align}\label{step3_b}
& \int\!\!\!\!\!\int_{Q} \frac{y_{t} v}{a} dt + \int\!\!\!\!\!\int_{Q} y_{x} v_{x} \, dx dt\notag\\
& = - \int\!\!\!\!\!\int_{Q} \left(\int_{0}^{1} K(t,x,\tau)y(t,\tau)\,d\tau \right) \frac{v}{a} \, dx dt+ \int\!\!\!\!\!\int_{Q}  \frac{h v}{a} \, dx dt.
\end{align}
Then, by a density argument, the above identity holds for all $v\in L^2(0, T; H_{\frac{1}{a}}^{1}(0,1))$.
Furthermore, from \cite[Theorem 11.4]{chipot2000}, we also have that $y\in C([0,T];  L^2_{\frac{1}{a}}(0,1))$.

It remains to show that $y(0,\cdot) = y_0$. From \eqref{step3_b}, we have
\begin{align*}
& - \int\!\!\!\!\!\int_{Q} \frac{y v_t}{a}\, dt + \int\!\!\!\!\!\int_{Q} y_{x} v_{x} \, dx dt \notag\\
& = - \int\!\!\!\!\!\int_{Q} \left(\int_{0}^{1} K(t,x,\tau)y(t,\tau)\,d\tau \right) \frac{v}{a} \, dx dt+ \int\!\!\!\!\!\int_{Q} \frac{h v}{a} \, dx dt + \int_0^1  \frac{y(0) v(0)}{a}\,dx,
\end{align*}
for all $v\in C^1([0,T]; H_{\frac{1}{a}}^1(0,1))$ with $v(T)=0$.		
Similarly, from \eqref{step3_a}, we also have
\begin{align*}
&- \int\!\!\!\!\!\int_{Q} \frac{y_m v_t}{a}\, dt + \int\!\!\!\!\!\int_{Q} y_{m,x} v_{x} \, dx dt\\
& = - \int\!\!\!\!\!\int_{Q} \left(\int_{0}^{1} K(t,x,\tau)y_m(t,\tau)\,d\tau \right) \frac{v}{a} \, dx dt+ \int\!\!\!\!\!\int_{Q}  \frac{h v}{a} \, dx dt + \int_0^1  \frac{y_m(0) v(0)}{a}\,dx.
\end{align*}
For $m=m_l$, since $y_{m_{l}}(0)\rightarrow y_0$ in $L^2_{\frac{1}{a}}(0,1)$, one deduces after passing to the limit that
\begin{align*}
&- \int\!\!\!\!\!\int_{Q} \frac{y v_t}{a}\, dt + \int\!\!\!\!\!\int_{Q} y_{x} v_{x} \, dx dt\\
& = - \int\!\!\!\!\!\int_{Q} \left(\int_{0}^{1} K(t,x,\tau)y(t,\tau)\,d\tau \right) \frac{v}{a} \, dx dt+ \int\!\!\!\!\!\int_{Q}  \frac{h v}{a} \, dx dt + \int_0^1  \frac{y_0 v(0)}{a}\,dx.
\end{align*}
Hence
$$<y_0-y(0),v(0)>_{L^2_{\frac{1}{a}}(0,1)}=0,\quad \forall\; v(0)\in H^1_{\frac{1}{a}}(0,1),$$
and by density of $H^1_{\frac{1}{a}}(0,1)$ in $L^2_{\frac{1}{a}}(0,1)$ we conclude that
$y(0) = y_0.$

\textbf{Step 4.} (Uniqueness).

Let $y$ be a solution of \eqref{problem_nondiv}. Proceeding as in \eqref{A7}, one can show an a priori estimate of the
form
\begin{align*}
\|y\|_{L^2(0,T;H_{\frac{1}{a}}^1(0,1))}^2 \leq C \left( \|y_0\|_{L^2_{\frac{1}{a}}(0,1)}^2 + \|h\|_{L^2_{\frac{1}{a}}(Q)}^2 \right).
\end{align*}
From this and linearity, uniqueness of solutions is obvious.
\end{proof}

%%%%%%%%%%%%%%%%%%%%%%%%%%%%%%%%%%%%%%%%%%%%%%%%%%%%%%%%%%%%%%%%%%%%%%%%%%%%%%%%%%%%%%%%%%%%%%%%%%%%%%%%%%%%%%%%%%%%%%%%%%%%%%%

\section{Null controllability for the nonhomogeneous systems}\label{section_null_nonhomo}
In this section, we will solve the null controllability problem for the following nonhomogeneous degenerate
systems in non divergence and divergence form:
\begin{equation}\label{nonhom_problem}
\left\{
\begin{array}{ll}
\displaystyle y_t - ay_{xx}  = f + 1_{\omega} u, & (t,x) \in Q,\\
y (t, 0) = y (t, 1) =0,& t \in (0,T),
\\
y(0,x)=y_0(x), & x \in (0,1)
\end{array}
\right.
\end{equation}
and
\begin{equation}\label{nonhom_problem_div}
\left\{
  \begin{array}{ll}
y_t - (ay_{x})_x = f + 1_{\omega} u, & (t,x) \in Q,\\
y(t, 1) = 0, & t \in (0,T),\\
\left\{
\begin{array}{ll}
y (t, 0) = 0,\quad \text{in the (WD) case},\\
(a y_{x})(t, 0)= 0, \quad \text{in the (SD) case},
\end{array}
\right.
& t \in (0,T),
\\
y(0,x)=y_0(x), & x \in (0,1),
  \end{array}
\right.
\end{equation}
where $f$ is a given source term.

First, let us recall the following well-posedness results.
\begin{proposition}\cite[Theorem 2.3]{FMpress}\label{prop-Well-posed_nonhom_nondiv}
For all $f\in L_{\frac{1}{a}}^2(Q)$, $ u \in L_{\frac{1}{a}}^2(Q)$ and $y_0 \in L_{\frac{1}{a}}^2(0,1)$, there exists a unique solution
\[
y \in \mathcal{W}:= L^2(0, T; H_{\frac{1}{a}}^1(0,1))\cap C([0, T]; L_{\frac{1}{a}}^2(0,1))
\]
of \eqref{nonhom_problem} and
\begin{align}\label{energy-nonhom1}
&\sup_{t\in[0,T]}\|y(t)\|^2_{L_{\frac{1}{a}}^2(0,1)} + \int_0^T \|y\|^2_{ H_{\frac{1}{a}}^1(0,1)} \, dt \notag\\
& \leq C \Big(\|y_0\|^2_{L_{\frac{1}{a}}^2(0,1)} +\|f\|^2_{L_{\frac{1}{a}}^2(Q)} + \|u\|^2_{L_{\frac{1}{a}}^2(Q_{\omega})} \Big),
\end{align}
for some positive constant $C$.
Moreover, if  $y_0\in H_{\frac{1}{a}}^{1}(0,1)$, then
\[
y \in \mathcal{Z}:= L^2(0, T; H_{\frac{1}{a}}^{2}(0,1))\cap H^1(0, T; L_{\frac{1}{a}}^2(0,1)) \cap  C([0,T]; H^1_{\frac{1}{a}}(0,1))
\]
and
\[
\begin{aligned}
&\sup_{t\in[0,T]}\|y(t)\|^2_{ H_{\frac{1}{a}}^1(0,1)} +\int_0^T \Big( \|y\|^2_{ H_{\frac{1}{a}}^2(0,1)}  + \|y_t\|^2_{ L^2_{\frac{1}{a}}(0,1)} \Big)\, dt\\
&\leq C \Big(\|y_0\|^2_{H_{\frac{1}{a}}^{1}(0,1)}+\|f\|^2_{L_{\frac{1}{a}}^2(Q)} + \|u\|^2_{L_{\frac{1}{a}}^2(Q_{\omega})} \Big),
\end{aligned}
\]
for some positive constant $C$.
\end{proposition}
%%%%%%%%%%%%%%%%%%%%%%%%%%%%%%%%%%%%%%%%%%
\begin{proposition}\cite[Theorem 2.6]{Alabau2006}\label{prop-Well-posed_nonhom_div}
For all $f\in L^2(Q)$, $u \in L^2(Q)$ and $y_0 \in L^2(0,1)$, there exists a unique solution
\[
y \in \mathcal{U}:= L^2(0, T; H_{a}^1(0,1))\cap C([0, T]; L^2(0,1))
\]
of \eqref{nonhom_problem_div} and
\[
\begin{aligned}
&\sup_{t\in[0,T]}\|y(t)\|^2_{L^2(0,1)} + \int_0^T \|y\|^2_{ H_{a}^1(0,1)} \, dt \\
& \leq C \Big(\|y_0\|^2_{L^2(0,1)} +\|f\|^2_{L^2(Q)} + \|u\|^2_{L^2(Q_{\omega})} \Big),
\end{aligned}
\]
for some positive constant $C$.
Moreover, if  $y_0\in H_{a}^{1}(0,1)$, then
\[
y \in \mathcal{V}:= L^2(0, T; H_{a}^{2}(0,1))\cap H^1(0, T; L^2(0,1)) \cap  C([0,T]; H^1_{a}(0,1))
\]
and
\[
\begin{aligned}
&\sup_{t\in[0,T]}\|y(t)\|^2_{ H_{a}^1(0,1)} +\int_0^T \Big( \|y\|^2_{ H_{a}^2(0,1)}  + \|y_t\|^2_{L^2(0,1)} \Big)\, dt \notag\\
&\leq C \Big(\|y_0\|^2_{H_{a}^{1}(0,1)}+\|f\|^2_{L^2(Q)} + \|u\|^2_{L^2(Q_{\omega})} \Big),
\end{aligned}
\]
for some positive constant $C$.
\end{proposition}

For our further results, a fundamental role is played by the following Hardy-Poincar\'e inequality.
\begin{proposition}\cite[Proposition 2.6]{CFR2008}\label{Prop_Hardy_nondiv}
Assume that Hypothesis \ref{Hypoth_1_nondiv} is satisfied. Then, there exists $C > 0$ such that
\begin{equation}\label{hardyineq_nondiv}
\int_{0}^{1}\frac{y^2}{a}\,dx \leq C \int_{0}^{1} y^2_x\,dx, \qquad \forall \;y \in H_0^1(0,1).
\end{equation}
\end{proposition}
As a consequence of Proposition \ref{Prop_Hardy_nondiv}, one can see that the norm $\|\cdot\|_{H_{\frac{1}{a}}^1}$ is equivalent to the norm of $H_0^1(0,1)$, and hence, the Banach spaces $ H_{\frac{1}{a}}^1(0, 1)$ and $H_0^1(0, 1)$ coincide.

In the next subsections, we prove new Carleman estimates for the adjoint parabolic problems
associated to the nonhomogeneous degenerate systems \eqref{nonhom_problem} and \eqref{nonhom_problem_div}, which will provide the null controllability property for the initial problems.
%%%%%%%%%%%%%%%%%%%%%%%%%%%%%%%%%%
\subsection{Carleman estimates for the problem in non divergence form}
In the following, we concentrate on the next adjoint problem associated  to \eqref{nonhom_problem} given by
\begin{equation}\label{adjoint_problem}
\left\{
\begin{array}{ll}
\displaystyle v_t +av_{xx}= g, & (t,x) \in Q,\\
v(t, 0) = v(t, 1) =0,& t \in (0,T),\\
v(T,x)=v_T(x), & x \in (0,1).
\end{array}
\right.
\end{equation}
Here, we assume that $g\in L_{\frac{1}{a}}^2(Q)$, while on the coefficient $a$ we make the following assumptions:
\begin{hypothesis}\label{Hypoth_2_nondiv}
The function $a\in C[0,1]\cap C^2(0,1]$ is such that $a(0) = 0$, $a>0$ on $(0, 1]$ and
\begin{enumerate}
\item there exists $\alpha \in [0,2)$ such that $$x a'(x) \leq \alpha a(x), \quad \forall \; x \in (0,1].$$
\item setting $\rho(x):= \frac{x a'(x)}{a(x)}$  there exists $\rho_{xx}$ for $x \in (0,1]$ and $c >0$ such that $$ \left(\frac{x a'(x)}{a(x)}\right)_{xx} \leq c \frac{1}{a(x)}, \quad \forall \; x \in (0,1].$$
\end{enumerate}
\end{hypothesis}
Observe that this assumption is more general than the one made in \cite{CFR2008} (see also \cite{yamamoto2020}). Moreover, Hypothesis \ref{Hypoth_2_nondiv}.1 implies Hypothesis \ref{Hypoth_1_nondiv}.

The goal of this subsection is to prove crucial estimates of Carleman's type for the nonhomogeneous system \eqref{adjoint_problem}. For this purpose, we introduce the weight function $\phi$ defined as follows
\begin{equation}\label{phi}
\phi(t,x):= \lambda\theta(t)(p(x) - \beta \|p\|_{L^{\infty}(0,1)}),
\end{equation}
where
\begin{equation}\label{weightfunc_deg_nodiv}
\begin{aligned}
p(x):= \int_{0}^{x}\frac{y}{a(y)} e^{y^2}\,dy ,\quad  \theta(t):=\frac{1}{[t(T-t)]^2},
\end{aligned}
\end{equation}
where $\beta>1$ and $\lambda>0$ are constants to be specified later. Observe that $ \phi(t,x)  <0$ for all $(t,x) \in Q$ and
$\phi(t,\cdot)  \rightarrow - \infty \, \text{ as } t \rightarrow
0^+, T^-$. Moreover, let us set:
\[
\psi(x):=\lambda(p(x) - \beta \|p\|_{L^{\infty}(0,1)}).
\]
One can check that $\psi$ is increasing and thus $-\beta\lambda \|p\|_{L^{\infty}(0,1)}=\psi(0) \le \psi (x) \le  \psi(1)= (1-\beta) \lambda\|p\|_{L^{\infty}(0,1)}$ for all $x \in [0,1]$.

Then, the following estimate holds.
\begin{theorem}\label{thm_Carl_bound}
Assume Hypothesis \ref{Hypoth_2_nondiv}. Then, there exist two positive constants $C$ and $s_0$, such that, every solutions $v$ of \eqref{adjoint_problem} in $\mathcal{Z}$ satisfies
\begin{align}\label{Carl_estimate_bound}
\int\!\!\!\!\!\int_{Q} \Big(s \theta   v_{x}^{2} &+ s^{3} \theta^3 \left(\frac{x}{ a}\right)^2 v^{2}\Big) e^{2s\phi}\,dx\,dt \notag \\ &\leq C \int\!\!\!\!\!\int_{Q} \frac{g^2}{a} e^{2s\phi}\,dx\,dt  + 2s C \int_0^T \theta(t) [x v_x^2 e^{2s\phi} ](t,1) dt,
\end{align}
for all $s \geq s_0$.
\end{theorem}
Note that, in \cite{CFR2008} Theorem \ref{thm_Carl_bound} is shown for $\lambda =1$, $\beta=2$, $\ds\theta(t)=\frac{1}{[t(T-t)]^4}$ and under a stronger assumption on $a$. However, by a simple adaptation of the proof, one can prove that the result remains true also in this case. For the reader's convenience we write the proof of Theorem \ref{thm_Carl_bound} in the Appendix.

Next, we introduce the following weight functions associated to the classical Carleman estimates:
$$\Phi(t,x):=\theta(t)\Psi(x)$$
where
\begin{equation}\label{weightfunc_nondeg}
\Psi(x) = e^{\rho \sigma} - e^{2 \rho \|\sigma\|_{\infty}}.
\end{equation}
Here $\rho>0$, $\sigma \in C^2([0,1])$ is such that $\sigma>0$ in $(0,1)$, $\sigma(0)=\sigma(1)=0$ and $\sigma_x (x)\neq 0$ in $[0,1]\setminus\tilde{\omega}$, being $\tilde{\omega}$ an arbitrary open subset of $\omega$.

By taking the parameter $\lambda$ such that
\begin{equation}\label{cond_rho}
\ds \lambda \geq \frac{e^{2 \rho \|\sigma\|_{\infty}} - 1 }{(\beta-1)\|p\|_{L^\infty(0,1)}},
\end{equation}
one can easily show that $\psi(x) \le \Psi(x)$ for all $x \in [0,1]$, and hence
\begin{equation}\label{compar_phi_eta}
\phi(t,x) \leq \Phi(t,x), \;\, \text{for all} \;(t,x) \in Q.
\end{equation}
We are going to derive from \eqref{Carl_estimate_bound} a Carleman estimate with locally distributed observation. To this aim, we need the following Caccioppoli's inequality (whose proof is postponed to the Appendix).
\begin{lemma}[Caccioppoli's inequality]\label{lemma_caccio}
Let $\omega_1$ and $\omega_2$ two open subintervals of $(0,1)$ such that $\omega_2 \Subset \omega_1 \Subset (0,1) .$ Let $\pi(t, x):=\theta(t) \eta(x),$ where $\theta$ is defined in \eqref{weightfunc_deg_nodiv} and $\eta \in C^{2}(0,1)$ is a strictly negative function. Then, there exist two strictly positive constants $C$ and $s_{0}$ such that, for all $s \geq s_{0}$,
\begin{equation}\label{inequality_caccio}
\begin{aligned}
&\int_{0}^{T}\!\!\!\int_{\omega_2} v_{x}^{2} e^{2 s \pi} dx dt \leq C\left(\int_{0}^{T}\!\!\!\int_{\omega_1} s^2 \theta^2 v^{2}e^{2 s \pi}\,dx\,dt + \int\!\!\!\!\!\int_{Q} g^{2} e^{2 s \pi}\,dx\,dt\right)\\
&\quad\leq C\left(\int_{0}^{T}\!\!\!\int_{\omega_1} s^2 \theta^2 \frac{v^{2}}{a}e^{2 s \pi}\,dx\,dt + \int\!\!\!\!\!\int_{Q} \frac{g^{2}}{a}e^{2s\pi}\,dx\,dt\right),
\end{aligned}
\end{equation}
for every solution $v$ of \eqref{adjoint_problem}.
\end{lemma}
With the aid of Theorem \ref{thm_Carl_bound} and Lemma \ref{lemma_caccio}, via suitable cut off functions, we can now prove the following $\omega$-local Carleman estimate for \eqref{adjoint_problem}.
\begin{theorem}\label{thm_Carl_local} Assume Hypothesis \ref{Hypoth_2_nondiv}.
There exist two positive constants $C$ and $s_0$, such that, every solutions $v$ of \eqref{adjoint_problem} in $\mathcal{Z}$ satisfies
\begin{align}\label{Carl_estimate_local}
\int\!\!\!\!\!\int_{Q} \Big(s \theta   v_{x}^{2} &+ s^{3} \theta^3 \left(\frac{x}{ a}\right)^2 v^{2}\Big) e^{2s\phi}\,dx\,dt \notag \\
&\leq C \Big(\int\!\!\!\!\!\int_{Q} \frac{g^2}{a} e^{2s\Phi}\,dx\,dt + \int\!\!\!\!\!\int_{Q_{\omega}} s^3 \theta^3 \frac{v^2}{a} e^{2s\Phi} \,dx dt\Big),
\end{align}
for all $s \geq s_0$.
\end{theorem}

\begin{proof}
Let us set
$\omega''= (\alpha'', \beta'') \Subset \omega'=(\alpha', \beta') \Subset \omega$
and consider a smooth cut-off function $\chi:[0,1] \rightarrow \mathbb{R}$ such that
$$
\left\{\begin{array}{ll}
0 \leq \chi(x) \leq 1, & \text { for all } x \in[0,1], \\
\chi(x)=1, & x \in (0,\alpha''), \\
\chi(x)=0, & x \in(\beta'',1).
\end{array}\right.
$$
We define $w:= \chi v$, where $v$ solves \eqref{adjoint_problem}. Then $w$ satisfies
\begin{equation}\label{equationchi}
\left\{
\begin{array}{lll}
w_t + aw_{xx} = G, & & (t,x)\in Q,\\
w(t, 1)=w(t, 0) = 0, & & t \in (0, T),\\
w(T,x)= \chi(x) v_{T}(x), & &  x \in  (0,1),
\end{array}
\right.
\end{equation}
where $G:= \chi g + a(\chi_{xx}v + 2\chi_{x} v_x)$.

Applying the Carleman inequality \ref{thm_Carl_bound} and using the fact that $w_x=0$ in the neighborhood of $x=1$, one has
\begin{align}\label{est-1-proof}
\int\!\!\!\!\!\int_{Q} \Big(s \theta   w_{x}^{2} &+ s^{3} \theta^3 \left(\frac{x}{ a}\right)^2 w^{2}\Big) e^{2s\phi}\,dx\,dt \leq C \int\!\!\!\!\!\int_{Q} \frac{G^2}{a} e^{2s\phi}\,dx\,dt.
\end{align}
Using the definition of $\chi$, in particular, the fact that  $\supp \chi_x, \chi_{xx} \subset \omega''$, we obtain
\begin{align*}
\int\!\!\!\!\!\int_{Q} \frac{G^2}{a} e^{2s\phi}\,dx\,dt & \leq 2 \int\!\!\!\!\!\int_{Q} a (\chi_{xx}v + 2\chi_{x} v_x)^2 e^{2s\phi}\,dx\,dt + 2 \int\!\!\!\!\!\int_{Q} \chi^2 \frac{g^2}{a} e^{2s\phi}\,dx\,dt\notag\\
& \leq C\Big( \int\!\!\!\!\!\int_{Q_{\omega''}} (v^2 + v_x^2) e^{2s\phi}\,dx\,dt + \int\!\!\!\!\!\int_{Q} \frac{g^2}{a} e^{2s\phi}\,dx\,dt\Big).
\end{align*}
Combining this last inequality and \eqref{est-1-proof}, via the Caccioppoli inequality \eqref{inequality_caccio} and using \eqref{compar_phi_eta}, it follows that
\begin{align}\label{est-2-proof}
\int\!\!\!\!\!\int_{Q} \Big(s \theta   w_{x}^{2} &+ s^{3} \theta^3 \left(\frac{x}{ a}\right)^2 w^{2}\Big) e^{2s\phi}\,dx\,dt \notag\\
& \leq C \Big(\int\!\!\!\!\!\int_{Q} \frac{g^2}{a} e^{2s\phi}\,dx\,dt + \int\!\!\!\!\!\int_{Q_{\omega'}}  s^2 \theta^2 \frac{v^{2}}{a}  e^{2s\phi}\,dx\,dt \Big)\notag\\
& \leq C \Big(\int\!\!\!\!\!\int_{Q} \frac{g^2}{a} e^{2s\Phi}\,dx\,dt + \int\!\!\!\!\!\int_{Q_{\omega'}}  s^3 \theta^3 \frac{v^{2}}{a}  e^{2s\Phi}\,dx\,dt\Big),
\end{align}
for $s$ large enough.

Now, consider $ z:= (1 - \chi) v$ and take $\tilde{\alpha}\in (0,\bar{\alpha})$. Then, $z$ satisfies
\begin{equation}\label{equationzeta}
\left\{
\begin{array}{lll}
z_t + az_{xx}= \tilde G, & & (t,x)\in (0,T) \times (\tilde{\alpha},1):=\tilde{Q},\\
z(t, 1)= z(t, \tilde{\alpha}) = 0, & & t \in (0, T),\\
z(T,x)= (1-\chi(x)) v_{T}(x), & &  x \in  (\tilde{\alpha},1),
\end{array}
\right.
\end{equation}
where $ \tilde G :=(1-\chi) g - a(\chi_{xx}v - 2\chi_{x} v_x)$.

Clearly, system \eqref{equationzeta} is a non degenerate problem; hence, by the classical Carleman estimate \cite[Lemma 1.2]{FI1996}, one has
\begin{equation}\label{est-3-proof}
%&\int\!\!\!\!\!\int_{Q} \Big( s\theta z_{x}^{2} + s^{3}\theta^3 z^{2}\Big) e^{2s\Phi}\,dx\,dt=
\int\!\!\!\!\!\int_{\tilde{Q}} \Big( s\theta z_{x}^{2} + s^{3}\theta^3 z^{2}\Big) e^{2s\Phi}\,dx\,dt
\leq C\Big(\int\!\!\!\!\!\int_{\tilde{Q}} \tilde G^2 e^{2s\Phi}\,dx\,dt+  \int\!\!\!\!\!\int_{Q_{\omega'}}  s^3 \theta^3 v^{2}  e^{2s\Phi}\,dx\,dt\Big).
\end{equation}
Using once again the fact that $\supp \chi_x, \chi_{xx} \subset \omega''$ and applying Caccioppoli's inequality, we obtain
\begin{align}\label{2est-3-proof}
\int\!\!\!\!\!\int_{\tilde{Q}} \tilde G^2 e^{2s\Phi}\,dx\,dt
&\leq C\Big( \int\!\!\!\!\!\int_{Q_{\omega''}} (v^2 + v_x^{2}) e^{2s\Phi}\,dx\,dt + \int\!\!\!\!\!\int_{\tilde{Q}} g^2 e^{2s\Phi}\,dx\,dt\Big)\notag\\
& \leq C\Big( \int\!\!\!\!\!\int_{Q_{\omega'}} s^2 \theta^2 \frac{v^{2}}{a}  e^{2s\Phi}\,dx\,dt + \int\!\!\!\!\!\int_{Q} \frac{g^2}{a} e^{2s\Phi}\,dx\,dt\Big).
\end{align}
In addition, via \eqref{compar_phi_eta} and by the definition of $z$, we have
\begin{align*}
\int\!\!\!\!\!\int_{Q} \Big( s\theta z_{x}^{2} + s^{3}\theta^3 \left(\frac{x}{a}\right)^{2} z^2 \Big) e^{2s\phi}\,dx\,dt
& \leq \int\!\!\!\!\!\int_{Q} \Big( s\theta z_{x}^{2} + s^{3}\theta^3 \left(\frac{x}{a}\right)^{2} z^2 \Big) e^{2s\Phi}\,dx\,dt\notag\\
&= \int\!\!\!\!\!\int_{\tilde{Q}} \Big( s\theta z_{x}^{2} + s^{3}\theta^3 \left(\frac{x}{a}\right)^{2} z^2\Big) e^{2s\Phi}\,dx\,dt \notag\\
&\leq C \int\!\!\!\!\!\int_{\tilde{Q}} \Big( s\theta z_{x}^{2} + s^{3}\theta^3 z^2 \Big) e^{2s\Phi}\,dx\,dt.
\end{align*}
Using \eqref{est-3-proof} and \eqref{2est-3-proof}, from the previous inequality, we find
\begin{align}\label{est-4-proof}
\int\!\!\!\!\!\int_{Q} \Big( s\theta z_{x}^{2} &+ s^{3}\theta^3 \left(\frac{x}{a}\right)^{2} z^2 \Big) e^{2s\phi}\,dx\,dt\notag \\
&\leq C\Big(\int\!\!\!\!\!\int_{Q} \frac{g^2}{a} e^{2s\Phi}\,dx\,dt + \int\!\!\!\!\!\int_{Q_{\omega'}}  s^3 \theta^3 \frac{v^{2}}{a} e^{2s\Phi}\,dx\,dt\Big),
\end{align}
for $s$ large enough.

Finally, observing that
\begin{equation*}
v^2=(w+z)^2\leq 2( w^2 + z^2)\qquad \text{and} \qquad v_x^2=(w_x+z_x)^2\leq 2(w_x^2 + z_x^2),
\end{equation*}
and combining \eqref{est-2-proof} and \eqref{est-4-proof}, one can easily deduce the desired estimate \eqref{Carl_estimate_local}.
\end{proof}

Now, we proceed to derive from \eqref{Carl_estimate_local} a new Carleman inequality with weight functions which do not blow up at $t=0$. This estimate will be the key tool to show that the null controllability for the nonhomogeneous problem \eqref{nonhom_problem} holds, imposing an exponential decay at the final time $t=T$ on the source term $f$. To this end, we define the following weight functions:
\begin{equation}\label{nu}
\nu(t)=  \left\{
\begin{array}{ll}
\theta(\frac{T}{2}), & t \in \left[0,\ds\frac{T}{2}\right], \\
\theta(t), & t \in \left[\ds\frac{T}{2},T\right],
\end{array}
\right.
\end{equation}

\begin{equation}\label{nwieght}
\begin{aligned}
&\tilde{\phi}(t,x):= \nu(t)\psi(x),\quad
\tilde{\Phi}(t,x):=\nu(t) \Psi(x),\quad \hat{\Phi}(t):=\displaystyle\max_{x\in[0,1]}\tilde{\Phi}(t,x), \\
&\hat{\phi}(t):=\displaystyle\max_{x\in[0,1]}\tilde{\phi}(t,x)\quad \text{and}\quad \check{\phi}(t):=\displaystyle\min_{x\in[0,1]}\tilde{\phi}(t,x).
\end{aligned}
\end{equation}

Then, the following modified Carleman estimate holds.

\begin{lemma}\label{lemma_modifiedcarl_nondiv} Assume Hypothesis \ref{Hypoth_2_nondiv} and let $T^*\in \left(\ds\frac{T}{2},T\right)$. Then, there exist two positive constants $C$ and $s_0$ such that every solution $v$ of \eqref{adjoint_problem} in $\mathcal{Z}$ satisfies, for all $s \geq s_0$
\begin{align}\label{modcarl_nondiv}
&\|e^{s\hat{\phi}(0)}v(0)\|_{L_{\frac{1}{a}}^2(0,1)}^2 +  \int\!\!\!\!\!\int_{Q} \frac{v^{2}}{a} e^{2s\tilde{\phi}}\,dx\,dt  \notag\\
&\leq C e^{2s[\hat{\phi}(0)-\check{\phi}(T^*)]} \Big(\int\!\!\!\!\!\int_{Q} \frac{g^2}{a} e^{2s\tilde{\Phi}}\,dx\,dt +
\int\!\!\!\!\!\int_{Q_{\omega}}  s^{3}\nu^3 \frac{v^{2}}{a} e^{2s\tilde{\Phi}}\,dx\,dt\Big).
\end{align}
\end{lemma}

\begin{proof}
Let $ \xi \in C^{\infty}([0, T])$ be a cut-off function such that
\begin{align}\label{Xi}
0\leq \xi \leq 1, \qquad \xi(t):=
\left\{
\begin{array}{lll}
1,   & \text{for} \; t\in\Big[0,\dfrac{T}{2}\Big],\\
0 , & \text{for} \; t\in \Big[T^*, T\Big].
\end{array}
\right.
\end{align}
We define $ w = \tilde{\xi} v$ ,  where $ \tilde{\xi} =  \xi  e^{s\hat{\phi}(0)} $ and $v$ solves \eqref{adjoint_problem}. Then $w$ satisfies
\begin{equation}\label{adjoint-w}
\left\{
\begin{array}{lll}
\displaystyle  w_t + a w_{xx} = \tilde g,  &  & (t, x) \in Q, \\
 w(t, 0) = w(t,1)=0, & & t \in (0, T),\\
w(T,x)= 0,  & & x \in  (0, 1),
\end{array}
\right.
\end{equation}
where $ \tilde g: = \tilde{\xi}_t v + \tilde{\xi} g.$

By the energy estimate \eqref{energy-nonhom1} applied to the above system, one can see that
\begin{equation*}
\sup_{t \in [0,T]} \|w(t)\|^2_{L_{\frac{1}{a}}^2(0,1)} \leq C \int\!\!\!\!\!\int_{Q}\frac{\tilde g^2}{a}\,dxdt.
\end{equation*}
Hence, there exists a
positive constant $C$ such that
\begin{equation}\label{energy-w}
\begin{aligned}
\int_0^1\frac{w^2(0)}{a}  \,dx  +\int\!\!\!\!\!\int_{Q}\frac{w^2}{a}\,dxdt &\leq C \int\!\!\!\!\!\int_{Q}\frac{\tilde g^2}{a}\,dxdt \\
&\leq C \int\!\!\!\!\!\int_{Q} \frac{1}{a}(\tilde{\xi}_t v + \tilde{\xi} g)^2  \,dx\,dt.
\end{aligned}
\end{equation}
From the definition of $\xi$ and the fact that $ \tilde\phi \leq \hat\phi(0) \; \text{in} \; Q$, one has
\begin{equation}\label{left-enrgy-1}
\int_0^1\frac{w^2(0)}{a}  \,dx =  \int_0^1\frac{v^2(0)}{a} e^{2s\hat\phi(0)}  \,dx
\end{equation}
and
\begin{equation}\label{left-enrgy-w-2}
\begin{aligned}
\int\!\!\!\!\!\int_{Q}\frac{w^2}{a}  \,dxdt   &=   \int_{0}^{T^*}\!\!\!\!\int_0^1\xi^2  e^{2s\hat{\phi}(0)} \frac{v^2}{a} dxdt \\
&\geq \int_{0}^{\frac{T}{2}}\!\!\!\!\int_0^1 \frac{v^2}{a} e^{2s\tilde\phi}  \,dxdt.
\end{aligned}
\end{equation}
Moreover, recalling that $  \xi\equiv 0$ in  $[T^*,T]$ and $ \supp \xi_t \subset [T/2, T^*]$, the second member in \eqref{energy-w} can be estimated as follows
\begin{align}\label{right-enrgy-w}
\int\!\!\!\!\!\int_{Q} \frac{1}{a} (\tilde{\xi}_t v + \tilde{\xi} g)^2  \,dx\,dt
& \leq C \Big( \int\!\!\!\!\!\int_{Q}   \xi_t^2  e^{2s\hat\phi(0)} \frac{v^2}{a}  \,dx\,dt + \int\!\!\!\!\!\int_{Q}   \xi^2 e^{2s\hat\phi(0)} \frac{g^2}{a}  \,dx\,dt \Big)\notag\\
& \leq C e^{2s\hat\phi(0)} \Big(\int_{\frac{T}{2}}^{T^*} \!\!\!\!\int_0^1   \frac{v^2}{a}  \,dx\,dt + \int_{0}^{T^*}\!\!\!\!\int_0^1  \frac{g^2}{a} \,dx\,dt \Big) \notag\\
&(\text{by \eqref{hardyineq_nondiv}})\notag\\
& \leq C e^{2s\hat\phi(0)} \Big( \int_{\frac{T}{2}}^{T^*} \!\!\!\!\int_0^1  v_x^2  \,dx\,dt + \int_{0}^{T^*}\!\!\!\!\int_0^1   \frac{g^2}{a} \,dx\,dt\Big).
\end{align}
Hence, from \eqref{energy-w}-\eqref{right-enrgy-w}, we get that
\begin{align*}
&\int_0^1\frac{v^2(0)}{a} e^{2s\hat\phi(0)}  \,dx + \int_{0}^{\frac{T}{2}}\!\!\!\!\int_0^1 \frac{v^2}{a} e^{2s\tilde\phi}  \,dxdt \notag\\
&  \leq C e^{2s\hat\phi(0)} \int_{\frac{T}{2}}^{T^*} \!\!\!\!\int_0^1  v_x^2  \,dx\,dt + C e^{2s\hat\phi(0)} \int_{0}^{T^*}\!\!\!\!\int_0^1   \frac{g^2}{a} \,dx\,dt.
\end{align*}
Recalling the definition of $\tilde \phi$, we observe that
\begin{equation*}
\check{\phi}\left(T^*\right) \leq \tilde{\phi} \;\quad \text{in}\quad \left(0,T^*\right)\times(0,1).
\end{equation*}
It results that
\begin{align}\label{estim_1}
&\int_0^1\frac{v^2(0)}{a} e^{2s\hat\phi(0)}  \,dx + \int_{0}^{\frac{T}{2}}\!\!\!\!\int_0^1 \frac{v^2}{a} e^{2s\tilde\phi}  \,dxdt \notag\\
&  \leq C e^{2s[\hat{\phi}(0)-\check{\phi}(T^*)]} \left( \int_{\frac{T}{2}}^{T^*} \!\!\!\!\int_0^1  v_x^2   e^{2s\tilde\phi}  \,dx\,dt +   \int_{0}^{T^*}\!\!\!\!\int_0^1  \frac{g^2}{a} e^{2s\tilde\phi} \,dx\,dt\right).
\end{align}
Using the fact that $\tilde\phi= \phi$ in $(T/2, T)\times(0,1)$, by the Carleman inequality \eqref{Carl_estimate_local}, we have that there exist two positive constants $C$ and $s_0$ such that
\begin{align*}
\int_{\frac{T}{2}}^{T^*} \!\!\!\!\int_0^1  v_x^2  e^{2s\tilde\phi}\,dx\,dt & = \int_{\frac{T}{2}}^{T^*} \!\!\!\!\int_0^1  v_x^2  e^{2s\phi}  \,dx\,dt\notag\\
 &\leq C \int_{0}^{T} \!\!\!\!\int_0^1 s \theta  v_x^2  e^{2s\phi}\,dx\,dt \notag\\
&\leq C\Big(\int\!\!\!\!\!\int_{Q} \frac{g^2}{a} e^{2s\Phi}\,dx\,dt  + \int\!\!\!\!\!\int_{Q_{\omega}} s^3 \theta^3 \frac{v^2}{a} e^{2s\Phi} \,dx dt\Big),
\end{align*}
for all $s \ge s_0$.

Using this last inequality in \eqref{estim_1} and taking into account the fact that $\tilde\phi, \Phi \leq \tilde\Phi$ in $Q$, we find
\begin{align}\label{estim_3}
&\|e^{s\hat{\phi}(0)}v(0)\|_{L_{\frac{1}{a}}^2(0,1)}^2 +  \int_{0}^{\frac{T}{2}}\!\!\!\!\int_0^1 \frac{v^{2}}{a} e^{2s\tilde{\phi}}\,dx\,dt \notag \\
&  \leq C e^{2s[\hat{\phi}(0)-\check{\phi}(T^*)]} \left ( \int\!\!\!\!\!\int_{Q} \frac{g^2}{a}  e^{2s\tilde\Phi}\,dx\,dt  + \int\!\!\!\!\!\int_{Q_{\omega}} s^3 \theta^3 \frac{v^2}{a} e^{2s\Phi}\,dx dt\right).
\end{align}

On the other hand, applying the Hardy-Poincar\'e inequality \eqref{hardyineq_nondiv} to $v e^{s\tilde\phi}$ and using the definition of $\tilde\phi$, one has
\begin{align}\label{estim_4}
\int_0^1 \frac{(ve^{s\tilde\phi})^2}{a}   \,dx & \leq C\int_0^1 (v e^{s\tilde\phi})_x^2   \,dx \notag\\
&\leq C\int_0^1 \left( v_x^2 + s^2 \nu^2 \left(\frac{x}{a}\right)^2 e^{2x^2} v^2 \right) e^{2s\tilde\phi}  \,dx \notag\\
& \leq C  \int_0^1    \left(s \nu   v_{x}^{2} + s^{3} \nu^3 \left(\frac{x}{ a}\right)^2 v^{2}\right) e^{2s\tilde\phi}  \,dx
\end{align}
for $s$ large enough. Thus, recalling that
$\nu=\theta$ and $\tilde\phi= \phi$ in $(T/2, T)\times(0,1)$, by \eqref{Carl_estimate_local}, we obtain that there exist two positive constants $C$ and $s_0$ such that
\begin{align}\label{estim_5}
\int_{\frac{T}{2}}^T\!\!\!\!\int_0^1 \frac{v^2}{a} e^{2s\tilde\phi}  \,dxdt  &\leq C  \int_{\frac{T}{2}}^T\!\!\!\!\int_0^1    \Big(s \nu v_{x}^{2} + s^{3} \nu^3 \left(\frac{x}{a}\right)^2 v^{2}\Big) e^{2s\tilde\phi}  \,dx dt \notag \\
&= C \int_{\frac{T}{2}}^T\!\!\!\!\int_0^1    \Big(s \theta   v_{x}^{2} + s^{3} \theta^3 \left(\frac{x}{ a}\right)^2 v^{2}\Big) e^{2s\phi}  \,dx dt \notag \\
& \leq C \int\!\!\!\!\!\int_{Q} \Big(s \theta   v_{x}^{2} + s^{3} \theta^3 \left(\frac{x}{ a}\right)^2 v^{2}\Big) e^{2s\phi}\,dx\,dt \notag \\
&\leq C\Big( \int\!\!\!\!\!\int_{Q} \frac{g^2}{a} e^{2s\Phi}\,dx\,dt  + \int\!\!\!\!\!\int_{Q_{\omega}} s^3 \theta^3 \frac{v^2}{a} e^{2s\Phi} \,dx\,dt\Big),
\end{align}
for all $s \ge s_0$.

Finally, notice that, for a positive constant $c$, the function $s \mapsto s^3 e^{-cs}$ is nonincreasing for $s$ large enough. Thus, since $ \nu \leq \theta$ in $(0,T)$,  we infer that
$$s^{3}\theta^3 e^{2s\Phi} \leq s^{3}\nu^3 e^{2s\tilde{\Phi}}, $$
for $s$ large enough.  This together with \eqref{estim_3} and \eqref{estim_5} imply \eqref{modcarl_nondiv}.
\end{proof}

As a consequence of Lemma \ref{lemma_modifiedcarl_nondiv} we will prove the null controllability  for the nonhomogeneous degenerate heat equation \eqref{nonhom_problem} with more regular solution. Such a result plays a fundamental role in obtaining null controllability for the nonlocal problem \eqref{problem_nondiv}. To this aim, following the arguments developed in \cite{FI1996,TG2016}, we define the following weighted space
\begin{equation*}
E_{s}:=\big\{ y \in \mathcal{Z}| \quad e^{-s\tilde{\Phi}} y \in L_{\frac{1}{a}}^2(Q) \big\}
\end{equation*}
endowed with the associated norm
\begin{align*}
\| y \|_{E_{s}}^2:=  \int\!\!\!\!\!\int_{Q} e^{-2s\tilde{\Phi}} \frac{y^2}{a} \, dx dt.
\end{align*}
Observe that, if we consider $y$ in $E_{s}$, then $y$ is continuous in time and satisfies
$$\int\!\!\!\!\!\int_{Q} e^{-2s\tilde{\Phi}} \frac{y^2}{a} \, dx dt<+\infty,$$
thus, from the definition of $\tilde{\Phi}$, in particular the fact that $\tilde{\Phi}<0$, it comes that $$ y(T,\cdot) = 0 \quad \text{in} \; (0,1).$$

In the sequel, by $s_0$ we shall denote the parameter given in Lemma \ref{lemma_modifiedcarl_nondiv}.
Following the classical approach presented in several works (see, e.g., \cite{cara2004,cara2006,gueye,TG2016}) and using the modified Carleman inequality \eqref{modcarl_nondiv}, we obtain the following null controllability result for \eqref{nonhom_problem}.

\begin{theorem}\label{Thm_nul_ctrl_nonhomog_nondiv}
Assume Hypothesis \ref{Hypoth_2_nondiv}. Let $T^*\in \left(\ds\frac{T}{2},T\right)$ and suppose that $e^{-s\tilde{\phi}}f \in L_{\frac{1}{a}}^2(Q)$ with $s\geq s_0$. Then, for any $y_0\in H_{\frac{1}{a}}^1(0,1)$, there exists a control function $u \in L_{\frac{1}{a}}^2(Q)$, such that the associated solution $y$ of \eqref{nonhom_problem} belongs to $E_{s}$.

Moreover, there exists a positive constant $C$ such that the couple $(y, u)$ satisfies
\begin{equation}\label{mainestimate_nondiv}
\begin{aligned}
&\int\!\!\!\!\!\int_{Q} \frac{y^2}{a} e^{-2s\tilde{\Phi}}\,dx\,dt +
\int\!\!\!\!\!\int_{Q_{\omega}}  s^{-3} \nu^{-3} \frac{u^{2}}{a} e^{-2s\tilde{\Phi}}\,dx\,dt \\
&\leq C e^{2s[\hat{\phi}(0)-\check{\phi}(T^*)]} \Big(\|f e^{-s\tilde{\phi}}\|^2_{L_{\frac{1}{a}}^2(Q)} +  \|y_0 e^{-s\hat{\phi}(0)}\|^2_{L^2_{\frac{1}{a}}(0,1)}\Big).
\end{aligned}
\end{equation}
\end{theorem}

\begin{proof}
Following the arguments in \cite{cara2004}, fixed $s \geq s_0$, let us consider the functional
\begin{equation}\label{extremal}
J(y,u)= \int\!\!\!\!\!\int_{Q} \frac{y^2}{a} e^{-2s\tilde{\Phi}}\,dx\,dt +
\int\!\!\!\!\!\int_{Q_{\omega}}  s^{-3}\nu^{-3}\frac{ u^{2}}{a} e^{-2s\tilde{\Phi}}\,dx\,dt ,
\end{equation}
where $u \in L^2_{\frac{1}{a}}(Q)$ and $y$ satisfies the following system
\begin{equation}\label{problem001}
\left\{
\begin{array}{ll}
y_t - a y_{xx}   =f + 1_{\omega} u, & (t,x) \in Q,\\
y(t,0)=y(t,1)=0, & t \in (0,T), \\
y(0,x)=y_0(x),\quad y(T,x)=0, & x \in (0,1).
\end{array}
\right.
\end{equation}
By standard arguments (see for instance \cite{Lions71}), $J$ attains its minimum at a unique point  $(\bar{y},\bar{u})$.

Let us denote by $ \mathcal L_a$ the parabolic operator
$$\mathcal L_a y:= y_t - a y_{xx}  \qquad \text{in}\quad Q.
$$
We are going to show that there exists a dual variable $\bar{v}$ such that the couple $(\bar{y},\bar{u})$ is characterized by
\begin{equation}\label{step0}
\left\{
\begin{array}{ll}
\bar{y}=e^{2s\tilde{\Phi}} \mathcal L_{a}^{\star} \bar{v}, &\quad \text{in}\quad Q,\\
\bar{u}= - s^3 \nu^3 e^{2s\tilde{\Phi}} \bar{v},  &\quad \text{in}\quad (0,T)\times \omega,\\
\bar{v}=0, &\quad \text{on}\quad (0,T)\times\{0,1\},
\end{array}
\right.
\end{equation}
where $\mathcal L_a^{\star}$ denotes the (formally) adjoint operator of $\mathcal L_a$.

Let $\mathcal{H}_0$ be the linear space
$$
\mathcal{H}_0=\big\{v \in C^{\infty}(\overline{Q}): v=0 \quad \text{on}\quad (0,T)\times\{0,1\}\big\},
$$
and let us define the following bilinear form:
\begin{equation*}
\kappa(v_1,v_2)= \int\!\!\!\!\!\int_{Q} \frac{e^{2s\tilde{\Phi}}}{a} \mathcal L_a^{\star} v_1 \mathcal L_a^{\star} v_2\,dx\,dt
+\int\!\!\!\!\!\int_{Q_{\omega}}  s^{3}\nu^3 \frac{e^{2s\tilde{\Phi}}}{a} v_1 v_2\,dx\,dt,\quad \forall \; v_1,v_2\in \mathcal{H}_0.
\end{equation*}
Since $\mathcal{H}_0\subseteq \mathcal{Z}$, we can apply the modified Carleman inequality \eqref{modcarl_nondiv} in $\mathcal{H}_0$, obtaining:
\begin{equation}\label{step000}
\|e^{s\hat{\phi}(0)}v(0)\|_{L_{\frac{1}{a}}^2(0,1)}^2 +  \int\!\!\!\!\!\int_{Q} \frac{v^{2}}{a} e^{2s\tilde{\phi}}\,dx\,dt
\leq C e^{2s[\hat{\phi}(0)-\check{\phi}(T^*)]} \kappa(v,v), \qquad \forall \; v \in \mathcal{H}_0.
\end{equation}
Observe that $\kappa(\cdot,\cdot)$ is a strictly positive and symmetric bilinear form in $\mathcal{H}_0$. Therefore, $\kappa(\cdot,\cdot)$ is a scalar product in $\mathcal{H}_{0}$.

Let $\mathcal{H}$ be the completion of $\mathcal{H}_0$ for the norm $\|v\|_{\mathcal{H}}:= (\kappa(v,v))^{1/2}$. Then $\mathcal{H}$ is a
Hilbert space for the scalar product $\kappa(\cdot,\cdot)$.

Now, we introduce the linear form $ \ell: \mathcal{H} \rightarrow \mathbb{R}$ given by
\begin{equation*}
\ell(v)=\int\!\!\!\!\!\int_{Q} \frac{f v}{a}\,dx\,dt + \int_{0}^{1}\frac{y_0 v(0)}{a}\,dx, \quad
\forall \; v\in \mathcal{H}.
\end{equation*}
Using the Cauchy-Schwarz inequality, in view of \eqref{step000}, it is clear that
\begin{align*}
|\ell(v)| &\leq \|f e^{-s\tilde{\phi}}\|_{L^2_{\frac{1}{a}}(Q)} \|v e^{s\tilde{\phi}}\|_{L^2_{\frac{1}{a}}(Q)}
+ \|y_0 e^{-s\hat{\phi}(0)}\|_{L^2_{\frac{1}{a}}(0,1)} \|v(0)e^{s\hat{\phi}(0)}\|_{L^2_{\frac{1}{a}}(0,1)} \\
&\leq C e^{s[\hat{\phi}(0)-\check{\phi}(T^*)]} \Big(\|f e^{-s\tilde{\phi}}\|_{L^2_{\frac{1}{a}}(Q)} +  \|y_0 e^{-s\hat{\phi}(0)}\|_{L^2_{\frac{1}{a}}(0,1)}\Big) \|v\|_{\mathcal{H}}.
\end{align*}
That is, $\ell$ is a bounded linear form on $\mathcal{H}$. Consequently, from the Lax-Milgram lemma, there exists a unique $\bar v$ in $\mathcal{H}$ such that
\begin{equation}\label{step001}
\kappa(\bar{v},v)= \ell(v),\quad \forall \;v\in \mathcal{H}.
\end{equation}
In addition, $\bar{v}$ satisfies the estimate
\begin{equation}\label{step002}
\|\bar{v}\|_{\mathcal{H}}\leq C e^{s[\hat{\phi}(0)-\check{\phi}(T^*)]}  \Big(\|f e^{-s\tilde{\phi}}\|_{L^2_{\frac{1}{a}}(Q)} +  \|y_0 e^{-s\hat{\phi}(0)}\|_{L^2_{\frac{1}{a}}(0,1)}\Big).
\end{equation}
Now, let us set
\begin{equation}\label{step003}
\begin{cases}
&\bar{y}=e^{2s\tilde{\Phi}} \mathcal L_a^{\star} \bar{v}\\
&\bar{u}= - 1_{\omega} s^3 \nu^3 e^{2s\tilde{\Phi}} \bar{v}.
\end{cases}
\end{equation}
With these definitions, from \eqref{step002}, it is readily seen that $\bar y$ and $\bar u$ satisfy
\begin{equation}\label{step004}
\begin{aligned}
&\int\!\!\!\!\!\int_{Q} \frac{\bar{y}^2}{a} e^{-2s\tilde{\Phi}}\,dx\,dt +
\int\!\!\!\!\!\int_{Q_{\omega}} s^{-3} \nu^{-3} \frac{\bar{u}^{2}}{a} e^{-2s\tilde{\Phi}}\,dx\,dt \\
&\leq C e^{2s[\hat{\phi}(0)-\check{\phi}(T^*)]}
\Big(\|f e^{-s\tilde{\phi}}\|^2_{L^2_{\frac{1}{a}}(Q)} +  \|y_0 e^{-s\hat{\phi}(0)}\|^2_{L^2_{\frac{1}{a}}(0,1)}\Big).
\end{aligned}
\end{equation}
Hence, \eqref{mainestimate_nondiv} holds.

It remains to check that $\bar{y}$ is the solution of \eqref{problem001} corresponding to $\bar{u}$. First of all, from \eqref{step004}, it is immediate that $\bar{y}\in E_{s}$ and $\bar{u}\in L^2_{\frac{1}{a}}(Q)$.

Let us denote by $\hat y$ the weak solution of \eqref{nonhom_problem}
with $u=\bar{u}$. Clearly, $\hat{y}$ is also the unique solution defined by transposition. This means that $\hat{y}$ is the unique function in $L_{\frac{1}{a}}^2(Q)$ satisfying
\begin{equation*}
\int\!\!\!\!\!\int_{Q} \frac{\hat{y} h}{a}\,dx\,dt= \int\!\!\!\!\!\int_{Q} 1_{\omega} \frac{\bar{u} v}{a}\,dx\,dt
+ \int\!\!\!\!\!\int_{Q} \frac{fv}{a}\,dx\,dt + \int_{0}^{1} \frac{y_0 v(0)}{a}\,dx,\quad \forall \;h \in L_{\frac{1}{a}}^2(Q),
\end{equation*}
where $v$ is the solution to the following system
\begin{equation*}
\left\{
\begin{array}{ll}
-v_t - a v_{xx} = h, & (t,x) \in Q,\\
v(t,0)=v(t,1)=0, & t \in (0,T), \\
v(T,x)=0 & x \in (0,1).
\end{array}
\right.
\end{equation*}
From \eqref{step003} and \eqref{step001}, we see that $\bar{y}$ also satisfies the last identity.
Consequently, $\bar{y}=\hat{y}$.

This ends the proof of Theorem \ref{Thm_nul_ctrl_nonhomog_nondiv}.
\end{proof}
\begin{remark}
Notice that Theorem \ref{Thm_nul_ctrl_nonhomog_nondiv} ensures the null controllability result for the nonhomogeneous degenerate equation \eqref{nonhom_problem} provided that the source terms $f$ satisfies
\begin{equation}\label{conditiononf}
e^{\frac{C_0}{(T-t)^2}} f \in L_{\frac{1}{a}}^2(Q)
\end{equation}
for some positive constant $C_0$. An analogous condition is required also in the nondegenerate case (see  \cite[Theorem 2.1, page 24]{FI1996}).
\end{remark}
%%%%%%%%%%%%%%%%%%%%%%%%%%%%%%%%%%%%%%%%%%%%%%%%%

\subsection{Carleman estimates for the problem in divergence form}

In this subsection, we present some Carleman estimates for the solutions to the associated adjoint problem
of \eqref{problem_div}, which will provide that the nonhomogeneous degenerate heat equation \eqref{nonhom_problem_div}
is null controllable. Thus, we consider the next adjoint problem
\begin{equation}\label{adjproblem_div}
\left\{
  \begin{array}{ll}
-z_t - (az_{x})_x = g, & (t,x) \in Q,\\
z(t, 1) = 0, & t \in (0,T),\\
\left\{
\begin{array}{ll}
z(t, 0) = 0,\quad \text{ for the (WD) case},\\
(a z_{x})(t, 0)= 0, \quad \text{for the (SD) case},
\end{array}
\right.
& t \in (0,T),
\\
z(T,x)=z_T(x), & x \in (0,1).
  \end{array}
\right.
\end{equation}
Here, we assume that $g\in L^2(Q)$, while on the diffusion coefficient $a$ we make the following assumptions:
\begin{hypothesis}\label{Hypoth_2_div}
The function $a\in C[0,1]\cap C^1(0,1]$ is such that $a(0) = 0$, $a>0$ on $(0, 1]$ and
\begin{enumerate}
\item (WD) CASE. There exists $\alpha \in [0,1)$, such that $x a'(x) \leq \alpha a(x)$, $\forall \;x \in [0,1]$.
\item (SD) CASE. There exists $\alpha \in [1,2)$, such that  $x a'(x) \leq \alpha a(x)$, $\forall \;x \in [0,1]$ and
\begin{equation*}
\left\{
\begin{array}{ll}
\exists \beta \in (1, \alpha],\, x \mapsto \dfrac{a(x)}{x^{\beta}}\quad \text{is nondecreasing near}\quad 0,\quad \text{if}\quad \alpha > 1, \\
\exists \beta \in (0, 1),\, x \mapsto \dfrac{a(x)}{x^{\beta}}\quad \text{is nondecreasing near}\quad 0,\quad \text{if}\quad \alpha=1.
\end{array}
\right.
\end{equation*}
\end{enumerate}
\end{hypothesis}

In this case, let us introduce the weight functions
\begin{equation}\label{weightfunc_div}
\varphi(t,x):=\theta(t)\Upsilon(x)\quad \text{where},\quad \Upsilon(x):= c\Big(\int_{0}^{x}\frac{y}{a(y)}\,dy - d\Big),
\end{equation}
and $\theta$ is the function defined in \eqref{weightfunc_deg_nodiv}.
The parameters $d$ and $c$ are
chosen such that $d> d^\star:=\displaystyle\sup_{[0,1]}\int_{0}^x\frac{y}{a(y)}\,dy$, $c\geq \frac{e^{2 \rho \|\sigma\|_{\infty}} -1}{d-d^\star}$, where  $\rho$ and $\sigma$ are defined in \eqref{weightfunc_nondeg}.

\begin{remark}
For this choice of the parameters $d, c$ and $\rho$, it is clear that the weight functions $\varphi$ and $\Phi$
satisfy the next inequality:
\begin{equation}\label{ineq_div}
\varphi(t,x) \leq \Phi(t,x) \quad \text{for every}\quad  (t,x)\in Q.
\end{equation}
\end{remark}

Then, the following Carleman inequality holds:
\begin{theorem}\cite[Theorem 3.3]{AHMS19}\label{Carleman01_div} Assume Hypothesis \ref{Hypoth_2_div}.
There exist two positive constants $C >0$ and $s_0 > 0$ such that, every solutions $z\in \mathcal{V}$ of \eqref{adjproblem_div} satisfies, for all
$s \geq s_0$
\begin{align}\label{Carl_div}
\int\!\!\!\!\!\int_{Q} \Big(s\theta a(x) z_{x}^{2} &+ s^{3}\theta^3 \frac{x^2}{a(x)}z^{2}\Big) e^{2s\varphi}\,dx\,dt \notag\\
&\leq
C  \Big(\int\!\!\!\!\!\int_{Q} g^2  e^{2s\Phi}\,dx\,dt + \int\!\!\!\!\!\int_{Q_{\omega}} s^{3}\theta^3 z^2 e^{2s\Phi}\,dx\,dt\Big).
\end{align}
\end{theorem}
\begin{remark}
Note that, in \cite{AHMS19} Theorem \ref{Carleman01_div} is shown for $\theta(t)=\frac{1}{t^4(T-t)^4}$. However, by \cite[Remark 1]{Hajjaj2013}, one can prove that
the result remains true also for $\theta=\frac{1}{t^2(T-t)^2}$.
\end{remark}

As for the non divergence case, in order to obtain a Carleman
inequality for solutions of problem \eqref{adjproblem_div} via weights not exploding at $t = 0$, let us consider the following functions:
\[
\tilde{\varphi}(t,x):=\nu(t) \Upsilon(x),\,
\hat{\varphi}(t):=\displaystyle\max_{x\in[0,1]}\tilde{\varphi}(t,x)\quad \text{and} \quad \check{\varphi}(t):=\displaystyle\min_{x\in[0,1]}\tilde{\varphi}(t,x).
\]
With the aid of Proposition \ref{prop-Well-posed_nonhom_div} and Theorem \ref{Carleman01_div}, similar arguments as in Lemma \ref{lemma_modifiedcarl_nondiv} lead to the following modified Carleman inequality.
\begin{lemma}\label{modifiedcarl_div}
Assume Hypothesis \ref{Hypoth_2_div}. There exist two positive constants $C$ and $s_0$ such that, every solutions $z \in \mathcal{V}$ to \eqref{adjproblem_div} satisfies, for all $s \geq s_0$
\begin{align}\label{modcarl_div}
&\|e^{s\hat{\varphi}(0)}z(0)\|_{L^2(0,1)}^2 +  \int\!\!\!\!\!\int_{Q} z^{2} e^{2s\tilde{\varphi}}\,dx\,dt  \notag\\
&\leq C e^{2s[\hat{\varphi}(0)-\check{\varphi}(\frac{5T}{8})]} \Big(\int\!\!\!\!\!\int_{Q} g^2 e^{2s\tilde{\Phi}}\,dx\,dt +
\int\!\!\!\!\!\int_{Q_{\omega}}  s^{3}\nu^3 z^{2} e^{2s\tilde{\Phi}}\,dx\,dt\Big).
\end{align}
\end{lemma}

Once we have got estimate \eqref{modcarl_div}, we are ready to solve the null controllability problem for \eqref{nonhom_problem_div}. In fact, arguing as in the non divergence case, it is possible to obtain the following result.
\begin{theorem}\label{firstresult_div}
Assume Hypothesis \ref{Hypoth_2_div}. Let $T>0$ and suppose that $e^{-s\tilde{\varphi}}f \in L^2(Q)$ with $s\geq s_0$. Then, for any $y_0\in H^{1}_{a}(0,1)$, there exists a control function $u \in L^2(Q)$ such that the associated solution $y\in \mathcal{V}$
of \eqref{nonhom_problem_div} is in the space
\begin{equation*}
\mathcal{E}_{s}:=\big\{y\in \mathcal{V}: e^{-s\tilde{\Phi}}y \in L^2(Q) \big\}.
\end{equation*}
Moreover, there exists a positive constant $C$ such that the following
estimate holds:
\begin{equation}\label{mainestimate_div}
\begin{aligned}
&\int\!\!\!\!\!\int_{Q} y^2 e^{-2s\tilde{\Phi}}\,dx\,dt +
\int\!\!\!\!\!\int_{Q_{\omega}}  s^{-3} \nu^{-3} u^{2} e^{-2s\tilde{\Phi}}\,dx\,dt \\
&\leq C e^{2s[\hat{\varphi}(0)-\check{\varphi}(\frac{5T}{8})]} \Big(\int\!\!\!\!\!\int_{Q} f^{2} e^{-2s\tilde{\varphi}}\,dx\,dt+ \|y_0e^{-s\hat{\varphi}(0)}\|_{L^2(0,1)}^2\Big).
\end{aligned}
\end{equation}
\end{theorem}

Also in this case, notice that $y$ being in the space $\mathcal{E}_s$, we readily have
$$
\int\!\!\!\!\!\int_{Q} y^2 e^{-2s\tilde{\Phi}}\,dx\,dt<\infty.
$$
Since the weight $\tilde{\Phi}$ blows up as $t \rightarrow T^{-}$, the boundedness of the above integral yields $y(T,x)=0$.

%%%%%%%%%%%%%%%%%%%%%%%%%%%%%%%%%%%%%%%%%%%%%%%%%%%%%%%%%%%%%%%%%%%%%%%%%%%%%%%%%%%%%%%%%%%%%%%%%%%%%%%%%%%%%%%%%%%%%%%%%%%%%%%%%%%%%%%%%%
\section{Null controllability for the nonlocal problem \eqref{problem_nondiv}}\label{section_null_nonlocal_nondiv}

The goal of this section is to study the null controllability property for the nonlocal degenerate equation \eqref{problem_nondiv}. Such a result relies on the null controllability result for the nonhomogeneous system \eqref{nonhom_problem}. First of all, we choose the parameters $\rho$, $\lambda$ and $\beta$ such that
\begin{equation}\label{crucialchoice}
\rho > \frac{\ln 2}{\|\sigma\|_{\infty}}, \quad\lambda \in \left[\frac{e^{2 \rho \|\sigma\|_{\infty}} -1}{(\beta - 1)\|p\|_{L^\infty(0,1)}}, \frac{2\big( e^{2 \rho \|\sigma\|_{\infty}} - e^{\rho \|\sigma\|_{\infty}} \big)}{(\beta - 1)\|p\|_{L^\infty(0,1)}}\right)\quad \text{and}\quad \beta=4.
\end{equation}
\begin{remark}Observe that, since the parameter $\rho$ in \eqref{weightfunc_nondeg} is such that
$\rho > \frac{\ln 2}{\|\sigma\|_{\infty}}, \quad $
one can show that the interval 
$$\displaystyle \left[\frac{e^{2 \rho \|\sigma\|_{\infty}} - 1}{(\beta - 1)\|p\|_{L^\infty(0,1)}}, \, \frac{2( e^{2 \rho \|\sigma\|_{\infty}} - e^{ \rho \|\sigma\|_{\infty}})}{(\beta - 1) \|p\|_{L^\infty(0,1)}}\right)$$ 
is nonempty. 
\end{remark}
Thus, one can show the following lemma which is needed in the sequel.
\begin{lemma}\label{cruciallemma} Let $T^* = (1+\varepsilon)\frac{T}{2}$ with $\varepsilon \in \left(0, \sqrt{1- \frac{2\sqrt{2}}{3}}\right)$. Then
\begin{equation}\label{crucialinequality}
2(\hat{\phi}(0)-\check{\phi}(T^*)+\hat{\Phi}(0))<0.
\end{equation}
\end{lemma}
\begin{proof}
Since the constant $\lambda\in  \Big[\frac{e^{2 \rho \|\sigma\|_{\infty}} -1}{(\beta - 1)\|p\|_{L^\infty(0,1)}}, \frac{2\big( e^{2 \rho \|\sigma\|_{\infty}} - e^{\rho \|\sigma\|_{\infty}} \big)}{(\beta - 1)\|p\|_{L^\infty(0,1)}}\Big)$, we immediately get:
\begin{equation}\label{compar_max_min}
2 \hat{\Phi}(t) \leq \hat{\phi}(t) \qquad \text{for every}\quad t\in (0,T).
\end{equation} 
On the other hand, using the fact that $\beta =4$, from \eqref{compar_max_min}, it comes that
\begin{align*}
2(\hat{\phi}(0)-\check{\phi}(T^*)+\hat{\Phi}(0))& \leq  3\hat{\phi}(0)- 2\check{\phi}(T^*) \\
& = - 3 (\beta - 1) \lambda \|p\|_{L^\infty(0,1)} \nu(0) + 2 \beta \lambda \|p\|_{L^\infty(0,1)} \nu(T^*)\\
&= \lambda \|p\|_{L^\infty(0,1)} (8 \nu(T^*)-  9 \nu(0)).
\end{align*}
Finally, by the definition of $T^*$, one can show that
\begin{align*}
8 \nu(T^*)-  9\nu(0) < 0,
\end{align*}
and the claim follows.
\end{proof}
In the following, in order to obtain our first main null controllability result for \eqref{problem_nondiv}, we make the following assumption on the kernel $K$.
\begin{hypothesis}\label{Hypoth_K_nondiv} Assume that the kernel $K$ is such that
\begin{equation}\label{condition_K_nondiv}
e^{\frac{c_0s}{(T-t)^2}} \int_{0}^{1}\int_{0}^{1}  \frac{K^2(t,x,\tau)}{a(x)}\,d\tau\,dx \in L^{\infty}(0,T),
\end{equation}
where $c_0:=8\lambda\|p\|_{L^\infty(0,1)} (4/T)^2 $, $s\ge s_0$ is fixed and $s_0$ is the same of Theorem \ref{Thm_nul_ctrl_nonhomog_nondiv}.

\end{hypothesis}

\begin{remark}
The assumption \eqref{condition_K_nondiv} will play a crucial role in the proof of the null controllability of \eqref{problem_nondiv}.
Observe that, using Remark \ref{remark_WSD}, thanks to this hypothesis, the kernel $K$, as a function of $t$ and $x$, should behave like
$$ K(t,x,\cdot) \sim e^{\frac{-C}{(T-t)^2}}$$
when $\frac{1}{a} \in L^1(0,1)$ and
$$ K(t,x,\cdot) \sim e^{\frac{-C}{(T-t)^2}} a^{1/4}(x)$$
when $\frac{1}{\sqrt a} \in L^1(0,1)$.
\end{remark}
The starting point to prove the null controllability property for the nonlocal degenerate equation \eqref{problem_nondiv} is to show the null controllability for the following control system
\begin{equation}\label{problem20_nondiv}
\left\{
\begin{array}{ll}
\displaystyle y_t - ay_{xx} + \int_{0}^{1} K(t,x,\tau)w(t,\tau)\,d\tau  = 1_{\omega} u, & (t,x) \in Q,\\
y (t, 0) = y (t, 1) =0,& t \in (0,T),
\\
y(0,x)=y_0(x), & x \in (0,1),
\end{array}
\right.
\end{equation}
for any $
w \in E_{s,M}:=\big\{w \in E_{s}|\quad \|e^{-s\tilde{\Phi}}w\|_{L_{\frac{1}{a}}^2(Q)}\leq M \big\},$
where $M>0$ is an arbitrary constant and $s$ is the one given in Hypothesis \ref{Hypoth_K_nondiv}. More precisely, as a first step we prove that this system is null controllable under the condition \eqref{condition_K_nondiv}; then, as a second step, we deduce the null controllability for the original problem applying a classical fixed point argument.

Notice that $E_{s,M}$  is a non empty, bounded, closed, and convex subset of $L_{\frac{1}{a}}^2(Q)$.

Our first main result is the following:
\begin{theorem}\label{thm_nul_ctrl_nonloc_nondiv}
Assume Hypotheses \ref{Hypoth_2_nondiv} and \ref{Hypoth_K_nondiv}. Then for any $y_0 \in H_{\frac{1}{a}}^1(0,1)$, there exists a control function 
$u\in L_{\frac{1}{a}}^2(Q)$ such that the associated solution $y\in \mathcal{Z}$ of \eqref{problem_nondiv} satisfies
\begin{equation}\label{ncresult_nondiv}
y(T,\cdot)=0 \qquad \text{in}\quad (0,1).
\end{equation}
\end{theorem}

\begin{proof}
Take $w \in E_{s, M}$; recalling the definition of $\tilde \phi$, by \eqref{condition_K_nondiv}, one has
\begin{align}\label{estim1_nodiv}
&\int\!\!\!\!\!\int_{Q} \frac{e^{-2s\tilde{\phi}}}{a}\Big(\int_{0}^{1} K(t,x,\tau)w(t,\tau)\,d\tau \Big)^2 \,dx\,dt\notag\\
&= \int\!\!\!\!\!\int_{Q} e^{\frac{c_0 s}{(T-t)^2}} \Big( \int_{0}^{1}  \frac{K(t,x,\tau)}{\sqrt{a(x)}} w(t,\tau)\,d\tau\Big)^2\,dx\,dt\notag \\
&\leq \max\limits_{\tau\in(0,1)} a(\tau) \int\!\!\!\!\!\int_{Q} e^{\frac{c_0 s}{(T-t)^2}} \Big( \int_{0}^{1}  \frac{K^2(t,x,\tau)}{a(x)}\,d\tau  \Big) \Big( \int_{0}^{1} \frac{w^2(t,\tau)}{a(\tau)}\,d\tau \Big) \,dx\,dt \notag \\
&\leq \max\limits_{\tau\in(0,1)} a(\tau) \sup\limits_{t\in (0,T)} \Big(  e^{\frac{c_0 s}{(T-t)^2}} \int_{0}^{1} \int_{0}^{1} \frac{K^2(t,x,\tau)}{a(x)}\,d\tau\, dx  \Big) \Big( \int\!\!\!\!\!\int_{Q} \frac{w^2(t,\tau)}{a(\tau)}\,d\tau\,dt \Big) \notag \\
&\leq C \int\!\!\!\!\!\int_{Q} \frac{w^2(t,x)}{a(x)}\,dx\,dt.
\end{align}
Hence, using the fact that $e^{2s\tilde{\Phi}} \leq 1$ in $Q$, we obtain
\begin{align*}
&\int\!\!\!\!\!\int_{Q}  \frac{e^{-2s\tilde{\phi}}}{a}\Big(\int_{0}^{1} K(t,x,\tau)w(t,\tau)\,d\tau\Big)^2 \,dx\,dt\notag\\
&\leq C  \int\!\!\!\!\!\int_{Q} e^{-2s\tilde{\Phi}} \frac{w^2}{a}\,dx\,dt \leq C M^2<+\infty.
\end{align*}
Thus, setting $f:= - \int_{0}^{1} K(t,x,\tau)w(t,\tau)\,d\tau$, one has that $ e^{-s\tilde{\phi}}f  \in L_{\frac{1}{a}}^2(Q)$ and, by Theorem \ref{Thm_nul_ctrl_nonhomog_nondiv}, we infer that system \eqref{problem20_nondiv} is null controllable, that is, $\forall\; y_0 \in H_{\frac{1}{a}}^1(0,1)$, there exists a control function $u \in L_{\frac{1}{a}}^2(Q)$ such that the associated solution $y$ of \eqref{problem20_nondiv} satisfies
$y(T, \cdot)=0$ in $(0,1)$. In addition, in this case, the control $u$ satisfies the estimate
\begin{equation*}
\int\!\!\!\!\!\int_{Q_{\omega}}  s^{-3} \nu^{-3} \frac{u^2}{a} e^{-2s\tilde{\Phi}}\,dx\,dt \leq C e^{2s[\hat{\phi}(0)-\check{\phi}(T^*)]} \Big(M^2 + \|y_0e^{-s\hat{\phi}(0)}\|_{L_{\frac{1}{a}}^2(0,1)}^2\Big),
\end{equation*}
where $T^*$ is defined in Lemma \ref{cruciallemma}.

Next, we aim to extend this controllability result to the nonlocal problem \eqref{problem_nondiv} through the Kakutani's fixed point Theorem (see for instance \cite[Theorem 2.3]{cara}). 
For any $w\in E_{s,M}$, consider the mapping $B: E_{s,M} \rightarrow 2^{E_{s}}$ defined by
\begin{align*}
B(w)=\displaystyle\Big\{&y\in E_s: y \text{ satisfies \eqref{problem20_nondiv} is} \; \text{such that} \;\;  y(T, \cdot)=0 \; \text{in} \; (0,1)  \; \text{and} \; \exists \; u \in L_{\frac{1}{a}}^2(Q)\; \text{so that} \\
&\qquad\int\!\!\!\!\!\int_{Q_{\omega}}  s^{-3} \nu^{-3} \frac{u^2}{a} e^{-2s\tilde{\Phi}}\,dx\,dt \leq C e^{2s[\hat{\phi}(0)-\check{\phi}(T^*)]} \Big(M^2 + \|y_0e^{-s\hat{\phi}(0)}\|_{L_{\frac{1}{a}}^2(0,1)}^2\Big)\Big\}.
\end{align*}
In the following, we are going to prove that $B$ has at least one fixed point in $E_{s,M}$. To do so, we will check that all the conditions to apply the aforementioned Theorem, with respect to the  $L^2$ topology, are fulfilled.

%We will check that all the conditions to apply such a theorem holds.

First of all, observe that $B(w)$ is nonempty, closed and convex in  $L_{\frac{1}{a}}^2(Q)$.
Next, let us prove that $B(E_{s,M}) \subset E_{s,M}$ for a sufficiently large $M$. Proceeding as in \eqref{estim1_nodiv}, using \eqref{mainestimate_nondiv} and \eqref{condition_K_nondiv}, we immediately find
\begin{align*}
&\int\!\!\!\!\!\int_{Q} \frac{y^2}{a} e^{-2s\tilde{\Phi}}\,dx\,dt +
\int\!\!\!\!\!\int_{Q_{\omega}}  s^{-3} \nu^{-3} \frac{u^2}{a} e^{-2s\tilde{\Phi}}\,dx\,dt \notag\\
&\leq   C e^{2s[\hat{\phi}(0)-\check{\phi}(T^*)]} \Big( \int\!\!\!\!\!\int_{Q} \frac{w^2}{a}\,dx\,dt + \|y_0e^{-s\hat{\phi}(0)}\|_{L_{\frac{1}{a}}^2(0,1)}^2\Big).
\end{align*}
Moreover, since $\tilde \Phi \leq \hat\Phi(0)$ in $Q$, it follows that
\begin{align}\label{estim_a}
&\int\!\!\!\!\!\int_{Q} \frac{y^2}{a} e^{-2s\tilde{\Phi}}\,dx\,dt +
\int\!\!\!\!\!\int_{Q_{\omega}}  s^{-3} \nu^{-3} \frac{u^2}{a} e^{-2s\tilde{\Phi}}\,dx\,dt \notag\\
&\leq  C e^{2s[\hat{\phi}(0)-\check{\phi}(T^*)+\hat{\Phi}(0)]} \Big( \int\!\!\!\!\!\int_{Q} e^{-2s\tilde{\Phi}} \frac{w^2}{a}\,dx\,dt \Big) + C e^{-2s\check{\phi}(T^*)} \|y_0\|_{L_{\frac{1}{a}}^2(0,1)}^2\notag\\
&\leq  C e^{2s[\hat{\phi}(0)-\check{\phi}(T^*)+\hat{\Phi}(0)]} M^2 
+ C e^{-2s\check{\phi}(T^*)} \|y_0\|_{L_{\frac{1}{a}}^2(0,1)}^2.
\end{align}
On the other hand, thanks to \eqref{crucialinequality} and by increasing the parameter $s_0$ if necessary, we can fix $s$ large enough so that
\begin{equation*}
C e^{2s[\hat{\phi}(0)-\check{\phi}(T^*)+\hat{\Phi}(0)]}<\frac{1}{2}.
\end{equation*}
Hence, for $M$ sufficiently large, we deduce that
\begin{align}\label{estim_b}
\int\!\!\!\!\!\int_{Q} \frac{y^2}{a} e^{-2s\tilde{\Phi}}\,dx\,dt &+
\int\!\!\!\!\!\int_{Q_{\omega}}  s^{-3} \nu^{-3} \frac{u^2}{a} e^{-2s\tilde{\Phi}}\,dx\,dt \notag\\
&  \leq   \frac{M^2}{2} + C e^{-2s\check{\phi}(T^*)} \|y_0\|_{L_{\frac{1}{a}}^2(0,1)}^2 \leq M^2.
\end{align}
Thus, $B$ maps $E_{s,M}$ into  itself, i.e., $B(E_{s,M}) \subset E_{s,M}$.

Now, let $\{w_n\}$ be a sequence in $E_{s,M}$. Then, by Proposition \ref{prop-Well-posed_nonhom_nondiv}, the associated solutions $\{y_n\}$ are bounded in $\mathcal{Z}$, and hence, by Theorem \ref{thm_compact_generale}, $B(E_{s,M})$ is relatively compact in $L^2(Q)$.

It remains to verify that, $B$ is upper-semicontinuous. At first,
note that, for any $w\in E_{s,M}$, one can find at least a control function $u \in L^2(Q)$ such that the associated solution $y$ belongs to $E_{s,M}$. Thus, for any $\{w_n\}$, we can find a sequence of controls $\{u_n\} \in L^2(Q)$ such that the associated solutions $\{y_n\}$ belongs to $L^2(Q)$.
Let $w_n \rightarrow w$ in $ E_{s,M}$ and $ y_n \in B(w_n)$ such that $y_n \rightarrow y$ in $L^2(Q)$. Our aim is to prove that $ y \in B(w)$.

Let $u_n$ be the sequence corresponding to the control function. Thus, by Proposition \ref{prop-Well-posed_nonhom_nondiv} and \eqref{estim_b}, it comes that on a subsequence (again denoted as $n$) we have the following convergences:
\begin{align*}
u_n \rightarrow u \quad &\text{weakly in} \; L^2(Q), \\
y_n \rightarrow \hat y \quad &\text{weakly in} \;L^2(0, T; H_{\frac{1}{a}}^{2}(0,1))\cap H^1(0, T; L_{\frac{1}{a}}^2(0,1))\\
& \text{since} \; \mathcal{Z}\, \text{is continuously embedded in} \; L^2(0, T; H_{\frac{1}{a}}^{2}(0,1))\cap H^1(0, T; L_{\frac{1}{a}}^2(0,1)),\\
&\text{strongly in\;} C([0,T]; L_{\frac{1}{a}}^2(0,1)),
\end{align*}
thanks to Theorem \ref{thm_compact_generale}.

Hence, we get that $y = \hat y \in L_{\frac{1}{a}}^2(Q)$, and passing to the limit in the following system
\begin{equation}\label{sys_n}
\left\{
\begin{array}{lll}
\displaystyle y_{n,t} - a y_{n,xx} + \int_{0}^{1} K(t,x,\tau)w_n(t,\tau)\,d\tau = 1_{\omega} u_n, &  & (t, x) \in Q,\\
y_n(t, 0)= y_n(t,1)= 0, & & t \in (0,T),  \\
y_n(0,x)= y_{0}(x),  & & x \in  (0,1)
\end{array}
\right.
\end{equation}
we deduce that the pair $(u, y)$ also satisfies system \eqref{problem20_nondiv}. Therefore, $B$ is upper semicontinuous and satisfies the assumptions of the Kakutani's fixed point Theorem. Hence, we can deduce  that it exists at least one $y$ such that $y\in B(y)$.
By definition of the mapping  $B$, this yields that there exists at least one pair $(u, y)$ satisfying the conditions of Theorem \ref{thm_nul_ctrl_nonloc_nondiv}.
Hence the claim follows.
\end{proof}
\begin{remark}
Recalling that the weight function $\tilde{\Phi}$ goes to $-\infty$ at $t=T$, the inequality \eqref{estim_b} implies that the control function $u$ that drives the solution of \eqref{problem_nondiv} to rest decay at the end of the time horizon $[0,T]$.
\end{remark}

As a consequence of the previous theorem and proceeding as in \cite{Allal2020}, or in \cite{AFS2020} one has the next result.
\begin{theorem}\label{Thm_null_Control_nonlocal_2}
Assume Hypotheses \ref{Hypoth_2_nondiv} and \ref{Hypoth_K_nondiv}.
Then, for any $y_0 \in L^2_{\frac{1}{a}}(0,1)$, there exists a control function $u\in L_{\frac{1}{a}}^2(Q)$ such that the associated solution $y\in \mathcal{W}$ of \eqref{problem_nondiv} satisfies
\begin{equation}\label{ncresult_nondiv'}
y(T,\cdot)=0 \qquad \text{in}\quad (0,1).
\end{equation}
\end{theorem}
%%%%%%%%%%%%%%%%%%%%%%%%%%%%%%%%%%%%%%%%%%%%%%%%%%%%%%%%%%%%%%%%%%%%%%%%%%%%%
\begin{proof}
Consider the following parabolic problem:
\begin{equation*}
\left\{
\begin{array}{lll}
w_t - aw_{xx} + \int_{0}^{1} K(t,x,\tau)w(t,\tau)\,d\tau  = 0,  &  & (t, x) \in \left(0, \ds\frac{T}{2}\right)\times (0, 1), \\
w(t, 0) = w(t,1)= 0, & & t \in\left(0, \ds\frac{T}{2}\right),  \\
w(0,x)= y_{0}(x),  & & x \in  (0, 1),
\end{array}
\right.
\end{equation*}
where $y_0\in L^2_{\frac{1}{a}}(0,1)$ is the initial condition in \eqref{problem_nondiv}.

By Theorem \ref{Thm_wellposed_nonloc_nondiv}, the solution of this system belongs to
$$ W^*_T := L^2\left(0, \frac{T}{2}; H_{\frac{1}{a}}^1(0,1)\right)\cap C\left(\left[0, \frac{T}{2}\right]; L_{\frac{1}{a}}^2(0,1)\right).$$
Then, there exists $t_0 \in (0, \frac{T}{2})$ such that
$w(t_0, \cdot) := \tilde{w}(\cdot) \in H_{\frac{1}{a}}^1(0,1)$.

Now, we consider the following controlled parabolic system:
\begin{equation*}
\left\{
\begin{array}{lll}
z_t - az_{xx} + \int_{0}^{1} K(t,x,\tau)z(t,\tau)\,d\tau  = 1_{\omega} h, &  & (t, x) \in (t_0, T) \times (0, 1), \\
z(t, 0) =z(t,1)= 0, & & t \in (t_0, T),  \\
z(t_0,x)= \tilde{w}(x),  & & x \in  (0, 1).
\end{array}
\right.
\end{equation*}
By Theorem \ref{thm_nul_ctrl_nonloc_nondiv} applied in $(t_0,T)\times(0,1)$,
we can see that there exists a control function  $h \in L^2((t_0, T) \times (0, 1))$ such that the associated solution
\begin{equation*}
z \in \mathcal{Z}^*:= L^2(t_0,, T; H_{\frac{1}{a}}^{2}(0,1))\cap H^1(t_0,, T; L_{\frac{1}{a}}^2(0,1))  \cap  C([t_0,,T]; H^1_{\frac{1}{a}}(0,1))
\end{equation*}
satisfies
$$z(T, \cdot) = 0 \qquad \text{in} \; (0, 1).$$
Finally, setting
\begin{align*}
y:=
\left\{
\begin{array}{lll}
w, & \text{in} \quad \big[0, t_0\big],\\
z, & \text{in} \quad \big[t_0, T\big]
\end{array}
\right.
 \quad  \text{and} \quad u:=
\left\{
\begin{array}{lll}
0, & \text{in} \quad \big[0, t_0\big],\\
h, & \text{in} \quad \big[t_0, T\big],
\end{array}
\right.
\end{align*}
one can prove that $y\in \mathcal{W}$ is the solution to the system \eqref{problem_nondiv} corresponding to $u$ and satisfies
$$y(T, \cdot) = 0 \qquad \text{in} \; (0, 1).$$
Hence, our assertion is proved.
\end{proof}

\begin{remark}\label{remark_control}
Notice that, if $\supp K \subset(0,T) \times \omega^2$, the null controllability of  \eqref{problem_nondiv} follows without  assuming Hypothesis \ref{Hypoth_K_nondiv}. Indeed, it is a consequence of the one for the degenerate heat equation
\begin{equation}\label{syst_v}
\left\{
\begin{array}{ll}
\displaystyle y_t - ay_{xx} = 1_{\omega} v, & (t,x) \in Q,\\
y (t, 0) = y (t, 1) =0,& t \in (0,T),
\\
y(0,x)=y_0(x), & x \in (0,1),
\end{array}
\right.
\end{equation}
which is proved in \cite[Theorem 4.5]{CFR2008} (see also Theorem \ref{Thm_nul_ctrl_nonhomog_nondiv} with $f=0$):  given $T>0$ and $y_0 \in L^2_{\frac{1}{a}}(0,1)$ there exists $v \in L_{\frac{1}{a}}^2(Q)$ such that the solution $y$ of \eqref{syst_v} satisfies $y(T,\cdot) = 0$ in $(0,1)$.

Now, taking $ u(t,x) = v(t,x) -  \int_0^1 K(t,x,\tau)y(t,\tau)\,d\tau $ and  recalling that $\supp K \subset(0,T) \times \omega^2$, we have that, the pair $(y,u)$ satisfies
\begin{equation*}
\left\{
\begin{array}{ll}
\displaystyle y_t - ay_{xx} + \int_{0}^{1} K(t,x,\tau)y(t,\tau)\,d\tau  = 1_{\omega} u, & (t,x) \in Q,\\
y (t, 0) = y (t, 1) =0,& t \in (0,T),
\\
y(0,x)=y_0(x), \quad y(T,x)=0  & x \in (0,1),
\end{array}
\right.
\end{equation*}
which yields the null controllability of \eqref{problem_nondiv}.

On the other hand, if the control function is supported in the whole domain, that is, $\omega = (0,1)$, once again the assumption \ref{Hypoth_K_nondiv} is not necessary.
Indeed, as above, we have that,  for any $y_0 \in L^2_{\frac{1}{a}}(0,1)$, there exists $h \in L_{\frac{1}{a}}^2(Q)$ such that the controlled problem
\begin{equation}\label{syst_h}
\left\{
\begin{array}{ll}
\displaystyle y_t - ay_{xx} = h, & (t,x) \in Q,\\
y (t, 0) = y (t, 1) =0,& t \in (0,T),
\\
y(0,x)=y_0(x), \quad y(T,x)=0 & x \in (0,1),
\end{array}
\right.
\end{equation}
has a solution (observe that in this case the  observability inequality
for the solutions of the adjoint system
\begin{equation}\label{sys-adj-remark}
\left\{
\begin{array}{lll}
\displaystyle - v_t - v_{xx}  = 0, &  & (t, x) \in Q, \\
v(t,0)=v(t,1)= 0 , & & t \in (0, T),  \\
v(T,x)= v_{T}(x),  & & x \in  (0, 1).
\end{array}
\right.
\end{equation}
is \begin{align}\label{obs_ineq}
\int_0^1   v^2(0,x)\frac{1}{a} \, dx &\leq C_T
\int\!\!\!\int_{Q}  v^2 \frac{1}{a}\,dx dt, \qquad \text{for all} \, v_T \in  L^2_{\frac{1}{a}}(0,1).
\end{align} Indeed, in order to prove \eqref{obs_ineq}, one can multiply as usual the equation of \eqref{sys-adj-remark} by $-\ds\frac{v}{a}$ and integrate by parts over $(0, 1)$).

Thus, by setting $ u(t,x) = h(t,x) -  \int_{0}^{1} K(t,x,\tau)y(t,\tau)\,d\tau $, this system can be rewritten as
\begin{equation*}
\left\{
\begin{array}{ll}
\displaystyle y_t - ay_{xx} + \int_{0}^{1} K(t,x,\tau)y(t,\tau)\,d\tau  =  u, & (t,x) \in Q,\\
y (t, 0) = y (t, 1) =0,& t \in (0,T),
\\
y(0,x)=y_0(x), \quad y(T,x)=0  & x \in (0,1),
\end{array}
\right.
\end{equation*}
and hence the null controllability of \eqref{problem_nondiv}.
\end{remark}

\begin{remark}
Finally, observe that thanks to Theorem \ref{Thm_nul_ctrl_nonhomog_nondiv} one can prove a  null controllability  result also for the memory system
\begin{equation}\label{memory}
\begin{cases}
\displaystyle y_t -a(x)  y_{xx} = \int\limits_0^t b(t,s,x)  y(s,x) \, ds + 1_{\omega} u, & (t, x) \in Q, \\
y(t,1)= y(t,0)=0,  & t \in (0, T),  \\
y(0,x)= y_{0}(x),   & x \in  (0, 1),
\end{cases}
\end{equation}
 as in \cite{Allal2020}, for the divergence case, or in \cite{AFS2020}, for the singular case. In particular the next theorem holds:
\begin{theorem}\label{Thm_null_Control_memory}
Assume Hypothesis \ref{Hypoth_2_nondiv} and suppose that $b\in L^\infty((0,T)\times Q)$ is such that
\begin{equation}
e^{\frac{c_0s}{(T-t)^2}} b^2(t,s,x) \in L^\infty((0,T)\times Q),
\end{equation}
where $c_0$ and $s$ are as in Hypothesis \ref{Hypoth_K_nondiv}.
Then, for any $y_0 \in L^2_{\frac{1}{a}}(0,1)$,
there exists $u \in L^2(Q)$ such that the associated solution $y \in \mathcal{W}$  of \eqref{memory} satisfies
$$ \quad y(T, \cdot) = 0 \qquad \text{in} \; (0, 1). $$
\end{theorem}

\end{remark}
%%%%%%%%%%%%%%%%%%%%%%%%%%%%%%%%%%%%%%%%%%%%%%%%%%%%%%%%%%%%%%%%%%%%%%%%%%%%%%%%%%%%%%%%%%%%
\section{Null controllability for the nonlocal problem \eqref{problem_div}}\label{section_null_nonlocal_div}

In this section, we pass to derive our second main result, which concerns the null controllability of the nonlocal degenerate heat equation \eqref{problem_div}.
Hence, in what follows, we assume that the function $K$ satisfies
\begin{equation}\label{conditiona}
e^{\frac{c_1 s}{(T-t)^2}}K \in L^{\infty}(Q\times (0,1)),
\end{equation}
where $c_1:=cd(\frac{4}{T})^2$ and $c,d$ are the constants defined in \eqref{weightfunc_div} and $s$ is the same of Lemma \ref{modifiedcarl_div}.

In fact, we are going to apply Theorem \ref{firstresult_div} to the nonlocal degenerate problem \eqref{problem_div}  and obtain
the following result.
\begin{theorem}\label{mainresult}
Assume Hypothesis \ref{Hypoth_2_div}. Let $T>0$ and assume that the function $K$ satisfies \eqref{conditiona}. Then for any $y_0\in H_{a}^{1}(0,1)$, there exists a control function $u\in L^2(Q)$ such that the associated solution $y\in \mathcal{V}$ of \eqref{problem_div} satisfies
\begin{equation}\label{ncresult}
y(T,\cdot)=0 \qquad \text{in}\quad (0,1).
\end{equation}
\end{theorem}

\begin{proof}
For our proof we are going to employ a fixed point strategy. For $R> 0$, we define
$$
\mathcal{E}_{s,R}=\big\{w \in \mathcal{E}_{s}:\, \|e^{-s\tilde{\Phi}}w\|_{L^2(Q)}\leq R \big\},
$$
which is a bounded, closed, and convex subset of $L^2(Q)$. For any $w \in \mathcal{E}_{s,R}$, let us consider the control problem
\begin{equation}\label{controlproblem_div}
\left\{
\begin{array}{ll}
y_t - (ay_{x})_x  + \int_{0}^{1}K(t,x,\tau)w(t,\tau)\,d\tau = 1_{\omega} u, & (t,x) \in Q,\\
y(t, 1) = 0, & t \in (0,T),\\
\left\{
\begin{array}{ll}
y (t, 0) = 0,\quad \text{in the (WD) case},\\
(a y_{x})(t, 0)= 0, \quad \text{in the (SD) case},
\end{array}
\right.
& t \in (0,T),
\\
y(0,x)=y_0(x), & x \in (0,1).
  \end{array}
\right.
\end{equation}
Now, from hypothesis \eqref{conditiona} we have that:
\begin{align}\label{r1}
&\int\!\!\!\!\!\int_{Q}  \Big( e^{-s\tilde{\varphi}} \int_{0}^{1} K(t,x,\tau)  w(t,\tau)\,d\tau\Big)^2 \,dx\,dt
\leq \int\!\!\!\!\!\int_{Q}  \int_{0}^{1} e^{\frac{2c_1 s}{(T-t)^2}} K^2(t,x,\tau) w^2(t,\tau)\,d\tau\,dx\,dt \notag\\
&\leq C \int\!\!\!\!\!\int_{Q} w^2\,dx\,dt \notag \leq C \Big(\displaystyle\sup_{(t,x) \in \overline{Q}} e^{2s\tilde{\Phi}} \Big) \int\!\!\!\!\!\int_{Q} e^{-2s\tilde{\Phi}} w^2\,dx\,dt \notag\leq C R^2<+\infty.
\end{align}
Therefore, from Lemma \ref{modifiedcarl_div} we have that \eqref{controlproblem_div} is null controllable, i.e., for any
$y_0\in H_{a}^{1}(0,1)$, there exists a control function $u\in L^2(Q)$ such that the associated
solution $y$ to \eqref{controlproblem_div} satisfies $y(T,\cdot)=0 \quad \text{in}\quad (0,1)$.

Notice that \eqref{controlproblem_div} is different from the original system \eqref{problem_div}, since the variable $y$
does not enter directly in the integral term. In order to conclude our proof and obtain the same controllability result for $w=y$,
we shall apply Kakutani's fixed point Theorem.

For any
$w \in \mathcal{E}_{s,R}$, we define the multivalued map $\Lambda: \mathcal{E}_{s,R} \subset \mathcal{E}_{s} \rightarrow 2^{\mathcal{E}_{s}}$ such that
\begin{align*}
\Lambda(w)=\displaystyle\Big\{&y: y\in \mathcal{E}_s\; \text{and there exists}\; u \in L^2(Q)\; \text{such that} \\
&\int\!\!\!\!\!\int_{Q_{\omega}}  s^{-3} \nu^{-3} u^{2} e^{-2s\tilde{\Phi}}\,dx\,dt
\leq C e^{2s[\hat{\varphi}(0)-\check{\varphi}(\frac{5T}{8})]} \Big(R^2+\int_{0}^{1}y_0^2 e^{-2s\hat{\varphi}(0)}\,dx\Big)\\
&y \text{ solves } \eqref{controlproblem_div}\Big\}.
\end{align*}
It is easy to check that $\Lambda(w)$ is a nonempty, closed, and convex subset of $L^2(Q)$.
Moreover, by \eqref{mainestimate_div} and \eqref{conditiona}, proceeding as before we have:
\begin{align*}
&\int\!\!\!\!\!\int_{Q} y^2 e^{-2s\tilde{\Phi}}\,dx\,dt +
\int\!\!\!\!\!\int_{Q_{\omega}}  s^{-3} \nu^{-3} u^{2} e^{-2s\tilde{\Phi}}\,dx\,dt \\
&\leq C e^{2s[\hat{\varphi}(0)-\check{\varphi}(\frac{5T}{8})]} \Big(\int\!\!\!\!\!\int_{Q}  e^{-2s\tilde{\varphi}}  \Big(\int_{0}^{1} K(t,x,\tau)  w(t,\tau)\,d\tau\Big)^2 \,dx\,dt + e^{-2s\hat{\varphi}(0)} \int_0^1 y_0^2\,dx\Big) \\
&\leq   C e^{2s[\hat{\varphi}(0)-\check{\varphi}(\frac{5T}{8})]} \Big( \int\!\!\!\!\!\int_{Q} w^2(t,x)\,dx\,dt + e^{-2s\hat{\varphi}(0)}  \int_0^1 y_0^2\,dx\Big)\\
&\leq  C e^{2s[\hat{\varphi}(0)-\check{\varphi}(\frac{5T}{8})]} \Big(\displaystyle\sup_{(t,x) \in \overline{Q}} e^{2s\tilde{\Phi}} \Big) \Big( \int\!\!\!\!\!\int_{Q} e^{-2s\tilde{\Phi}(t,x)} w^2(t,x)\,dx\,dt \Big) + C e^{-2s\check{\varphi}(\frac{5T}{8})} \int_0^1 y_0^2\,dx.
\end{align*}
Since $\hat \varphi(0) \le \hat \Phi(0)$ and $ \tilde{\Phi} \leq \hat{\Phi}(0)$ in $Q$, then we get
\begin{equation}\label{starstima}
\begin{aligned}
&\int\!\!\!\!\!\int_{Q} y^2 e^{-2s\tilde{\Phi}}\,dx\,dt +
\int\!\!\!\!\!\int_{Q_{\omega}}  s^{-3} \nu^{-3} u^{2} e^{-2s\tilde{\Phi}}\,dx\,dt \\
&\leq  C e^{s[2\hat{\varphi}(0)-2\check{\varphi}(\frac{5T}{8})+2\hat{\Phi}(0)]} \int\!\!\!\!\!\int_{Q} e^{-2s\tilde{\Phi}(t,x)} w^2(t,x)\,dx\,dt
+ C e^{-2s\check{\varphi}(\frac{5T}{8})} \int_0^1 y_0^2\,dx  \\
&\leq  C e^{s[4\hat{\Phi}(0)-2\check{\varphi}(\frac{5T}{8})]} R^2 + C e^{-2s\check{\varphi}(\frac{5T}{8})} \int_0^1  y_0^2\,dx.
\end{aligned}
\end{equation}
Now, choosing the constant $c$ (see \eqref{weightfunc_div}) in the interval
$$
\left(\frac{e^{2 \rho \|\sigma\|_{\infty}} -1}{d-d^{\star}}, \frac{16}{15}\frac{e^{2 \rho \|\sigma\|_{\infty}} -e^{\rho \|\sigma\|_{\infty}}}{d-d^{\star}}\right),
$$
which is not empty for $\rho$ sufficiently large, we have
\begin{align*}
2\hat{\Phi}(0)-\check{\varphi}\left(\frac{5T}{8}\right)
&=\left(\frac{4}{T^2}\right)^2\Big[2(e^{\rho\|\sigma\|_{\infty}}-e^{2\rho\|\sigma\|_{\infty}}) + c d\left (\frac{16}{15}\right)^2 \Big]\\
&< \left(\frac{4}{T^2}\right)^2 \left(-2 +  \frac{d}{d-d^{\star}} \left(\frac{16}{15}\right)^{3} \right) (e^{2\rho\|\sigma\|_{\infty}}-e^{\rho\|\sigma\|_{\infty}}).
\end{align*}
Therefore, taking the parameter $d$ defined in \eqref{weightfunc_div} in such a way that  $d>4d^{\star}$, we infer that
$$
2\hat{\Phi}(0)-\check{\varphi}(\frac{5T}{8})<0.
$$
%===================================
Hence for $s$ sufficiently large, increasing the parameter $s_0$ if necessary, we obtain
\begin{equation*}
\int\!\!\!\!\!\int_{Q} y^2 e^{-2s\tilde{\Phi}}\,dx\,dt +
\int\!\!\!\!\!\int_{Q_{\omega}}  s^{-3} \nu^{-3} u^{2} e^{-2s\tilde{\Phi}}\,dx\,dt
\leq \frac{1}{2} R^2 + C e^{-2s\check{\varphi}(\frac{5T}{8})} \int_0^1  y_0^2\,dx.
\end{equation*}
Thus, for  $s$ and $R$ large enough, we obtain
\begin{equation*}
\int\!\!\!\!\!\int_{Q} y^2 e^{-2s\tilde{\Phi}}\,dx\,dt
\leq R^2.
\end{equation*}
It follows that $\Lambda(\mathcal{E}_{s,R}) \subset \mathcal{E}_{s,R}$.

Let $\{w_n\}$ be a sequence of $\mathcal{E}_{s,R}$. The regularity assumption on $y_0$
and Proposition \ref{prop-Well-posed_nonhom_div}, imply that the corresponding solutions $\{y_n\}$ are bounded in  $H^1\big(0,T; L^2(0,1)\big) \cap L^2\big(0,T;D(A)\big)$.
Therefore, by the Aubin-Lions Theorem \cite{simon}, $\Lambda(\mathcal{E}_{s,R})$ is a relatively compact subset of $L^2(Q)$.

In order to conclude, we have to prove that $\Lambda$ is upper-semicontinuous under the $L^2$ topology.
Notice that, for any $w\in \mathcal{E}_{s,R}$, we have at least one control $u \in L^2(Q)$ such that the corresponding
solution $y$ belongs to $\mathcal{E}_{s,R}$.
Hence, taking $\{w_n\}$ a sequence in $\mathcal{E}_{s,R}$, we can find a sequence of controls $\{u_n\}$ such that the corresponding solutions
$\{y_n\}$ is in $L^2(Q)$. Thus, let $\{w_n\}$ be a sequence satisfying $w_n \rightarrow w$ in $\mathcal{E}_{s,R}$ and
$y_n \in \Lambda(w_n)$ such that $y_n \rightarrow y$ in $L^2(Q)$. Our aim is to prove that
$y \in \Lambda(w)$.
For every $n$, we have a control $u_n \in L^2(Q)$ such that the system
\begin{equation}\label{div_sys_n}
\left\{
\begin{array}{lll}
\displaystyle y_{n,t} - (ay_{n,x})_x + \int_0^1 K(t,x,\tau) w_n(t,\tau)\,d\tau = 1_{\omega} u_n, &  & (t, x) \in Q,\\
y_n(t, 1) = 0, & t \in (0,T),\\
\left\{
\begin{array}{ll}
y_n(t, 0) = 0,\quad \text{in the (WD) case},\\
(a y_{n,x})(t, 0)= 0, \quad \text{in the (SD) case},
\end{array}
\right.
& t \in (0,T),
\\
y_n(0,x)= y_{0}(x),  & & x \in  (0,1)
\end{array}
\right.
\end{equation}
has a least one solution $y_n\in L^2(Q)$ that satisfies
\begin{equation*}
y_n(T,\cdot)=0 \qquad \text{in}\quad (0,1).
\end{equation*}
By the regularity of the solutions (see Proposition \ref{prop-Well-posed_nonhom_div}) and \eqref{starstima}, it follows (at least for a subsequence) that
\begin{align*}
u_n \rightarrow u \quad &\text{weakly in} \; L^2(Q),\notag \\
y_n \rightarrow y \quad &\text{weakly in} \; H^1\big(0,T; L^2(0,1)\big) \cap L^2\big(0,T;D(A)\big),\\
&\text{strongly in} \; C(0, T; L^2(0,1)).
\end{align*}
Passing to the limit in \eqref{div_sys_n}, we obtain a control $u\in L^2(Q)$ such that the corresponding solution $y$ to \eqref{controlproblem_div}
satisfies \eqref{ncresult}. This shows that $y\in \Lambda(w)$ and, therefore, the map $\Lambda$ is upper-semicontinuous.

Hence, all the assumptions of Kakutani's fixed point Theorem are fulfilled and
we infer that there is at least one $y \in \mathcal{E}_{s,R}$ such that $y\in \Lambda(y)$.
By the definition of $\Lambda$, this implies that there exists at least one pair $(y,u)$ satisfying the conditions of Theorem \ref{mainresult}.
The uniqueness of $y$ follows by Theorem \ref{Thm_wellposed_nonloc_div}.
Therefore, our assertion is proved.

\end{proof}

%%%%%%%%%%%%
%%%%%%%%%%%%%
%%%%%%%%%%%%%

\section{Appendix}\label{Appendix}
\subsection{Proof of Theorem \ref{thm_Carl_bound}}
As in \cite{CFR2008}, we define, for
$s > 0$, the function
\[
w(t,x) := e^{s \phi(t,x)} v(t,x)
\]
where $v$ is the solution of \eqref{adjoint_problem} in $\mathcal{Z}$. Observe that,
since $v\in \mathcal{Z}$, $w\in \mathcal{Z}$ and $w$ satisfies
\begin{equation}\label{1'}
\begin{cases}
(e^{-s\phi}w)_t + a(x)(e^{-s\phi}w)_{xx}  =g(t,x), & (t,x) \in
Q,
\\
w(0, x)= w(T,x)= 0, & x \in (0,1),
\\
w(t,0)= w(t, 1)= 0, & t \in (0, T).
\end{cases}
\end{equation}
Defining $Lv:= v_t + av_{xx}$ and $L_sw:=
e^{s\phi}L(e^{-s\phi}w)$, the equation of \eqref{1'} can be
recast as follows
\begin{center}
$
L_sw =e^{s\phi}g.
$
\end{center}
In particular, the following equality holds:
\begin{proposition}(see \cite[Proposition 3.1]{yamamoto2020})
The operator $L_sw$ can be rewritten as
\[
L_sw =L^+_sw + L^-_sw,
\]
where $L^+_s$ and $L^-_s$ denote the (formal) selfadjoint and skewadjoint parts of $L_s$. In this case
\[
\begin{cases}
L^+_sw := aw_{xx}
 - s \phi_t w + s^2a \phi_x^2 w,
\\[5pt]
L^-_sw := w_t -2sa\phi_x w_x -
 sa\phi_{xx}w.
 \end{cases}
\]
\end{proposition}
Moreover, set as usual $<u, v>_{L^2_{\frac{1}{a}}(Q)} \::=\;
$\small$\displaystyle\int\!\!\!\!\!\int_{Q}uv\frac{1}{a}dxdt$\normalsize,
\;one has
\begin{equation} \label{Dis1}
\|L^+_sw \|^2_{L^2_{\frac{1}{a}}(Q)} + \:\|L^-_sw\|^2_{L^2_{\frac{1}{a}}(Q)} +
2<L^+_sw, L^-_sw>_{L^2_{\frac{1}{a}}(Q)} =\;
\|ge^{s\phi}\|^2_{L^2_{\frac{1}{a}}(Q)}.
\end{equation}
In \cite[Lemma 3.8]{CFR2008} it is proved that  the scalar product $<L^+_sw,L^-_sw>_{L^2_{\frac{1}{a}}(Q)}$ can be rewritten as the sum of ditributed and boundary terms in the following way:
\begin{equation}\label{D&BT}
\left.
\begin{aligned}
<L^+_sw,L^-_sw>_{L^2_{\frac{1}{a}}(Q_{{}_1})}
\;&=\;
s\int\!\!\!\!\!\int_{Q}(a\phi_{xx}+(a\phi_x)_x)w_x^2dxdt
\\[3pt]&+ s^3 \int\!\!\!\!\!\int_{Q}\phi_x^2(a\phi_{xx}+(a\phi_x)_x)w^2dxdt
\\[3pt]
&
-2 s^2\int\!\!\!\!\!\int_{Q}\phi_x\phi_{xt}w^2dxdt
 +\frac{s}{2}\int\!\!\!\!\!\int_{Q}\frac{\phi_{tt}}{a}w^2dxdt
\\[3pt] &-\frac{s}{2}\int\!\!\!\!\!\int_{Q}(a\phi_{xx})_{xx}w^2 dxdt
\end{aligned}\right\}\;\text{\{D.T.\}}
\end{equation}
%\vspace{-5pt}
\begin{equation}\nonumber
\hspace{55pt}
\text{\{B.T.\}}\;\left\{
\begin{aligned}
&
-\frac{1}{2}\int_0^{1}\Big[w_x^2\Big]_0^Tdx
+\int_0^T\Big[w_xw_t\Big]_0^{1}dt
\\[3pt]&+ \frac{s}{2} \int_0^T \Big[(a\phi_{xx})_xw^2\Big]_0^{1}dt
- s \int_0^T\Big[ a \phi_xw_x^2\Big]_0^{1}dt\\[3pt]&
- s \int_0^T\!\!\Big[ a\phi_{xx}ww_x\Big]_0^{1}dt
+\frac{1}{2} \int_{0}^{1}\!\!\Big[ (s^2 \phi_x^2 - s \frac{\phi_t}{a})w^2\Big]_0^Tdx
\\[3pt]
&
- s \int_0^T\Big[  (s^2 a \phi_x^3- s\phi_x\phi_t)w^2\Big]_0^{1}dt.
\end{aligned}\right.
\end{equation}
 Moreover, as in \cite[Lemma 3.9]{CFR2008}, one has that
the boundary terms in \eqref{D&BT} become
\begin{equation}\label{BT1}
\begin{aligned}
\{B.T.\}=-se\int_0^T\theta(t)w_x^2(t,1)dt.
\end{aligned}
\end{equation}

Now, we can prove the Carleman estimate \eqref{Carl_estimate_bound}.
\begin{proof}[Proof of Theorem \ref{thm_Carl_bound}]Using the definition of $\phi$, the distributed terms of
\:$<L^+_sw, L^-_sw>_{L^2_{\frac{1}{a}}(Q)}$ take the form
\begin{equation}\label{02}
\begin{aligned}
\big\{D.T.\big\}\; &=
\quad s \lambda\int\!\!\!\!\!\int_{Q}\theta\Big(2-\frac{xa_x}{a}+ 4x^2\Big)e^{x^2}w_x^2dxdt
\\
&+s^3\lambda^3\int\!\!\!\!\!\int_{Q}\theta^3\Big(\frac{x}{a}\Big)^2\Big(2-\frac{xa_x}{a}+ 4x^2\Big)e^{3x^2}w^2dxdt
\\
&-2 s^2\lambda^2\int\!\!\!\!\!\int_{Q}\theta\dot{\theta}\Big(\frac{x}{a}\Big)^2e^{2x^2}w^2dxdt
+\frac{s\lambda}{2} \int\!\!\!\!\!\int_{Q}\frac{\ddot{\theta}}{a}\Big(p- 
\beta \|p\|_{L^{\infty}(0,1)}\Big) w^2dxdt
\\
&+\frac{s\lambda}{2} \int\!\!\!\!\!\int_{Q}\theta
\Big(\frac{xa_x}{a}\Big)_{xx}e^{x^2}w^2dxdt +2s\lambda\int\!\!\!\!\!\int_{Q}\theta  x\Big(\frac{xa_x}{a}\Big)_{x}
e^{x^2}w^2dxdt
\\
& +s\lambda\int\!\!\!\!\!\int_{Q}\theta \left( (1+2x^2)\Big(\frac{xa_x}{a}\Big)
-(3+12x^2+4x^4)\right)e^{x^2}w^2dxdt.
\end{aligned}
\end{equation}
Because of Hypothesis \ref{Hypoth_2_nondiv}, we have
\[
 2-\frac{xa_x}{a}+ 4x^2\ge  2-\frac{xa_x}{a}\ge 2-\alpha>0\quad \forall\; x\in(0,1],
\]
\[
\begin{aligned}
&\left|\left(2x\Big(\frac{xa_x}{a}\Big)_{x}
+(1+2x^2)\Big(\frac{xa_x}{a}\Big) - (3+12x^2+4x^4)\right)
\right|\le
\\[3pt]
&  \left(2\norm{x\Big(\frac{xa_x}{a}\Big)_{x}}_{L^{\infty}(0,1)} \!\!\!+
3\norm{\Big(\frac{xa_x}{a}\Big)}_{L^{\infty}(0,1)}\!\!\! +
19)\right) =: C_1
\end{aligned}
\]
and
\[
\frac{1}{2}\left|\Big(\frac{xa_x}{a}\Big)_{xx}e^{x^2} \right| \le c \frac{e}{a}.
\]
Hence, by the previous estimates and \eqref{02}, we obtain
\begin{equation}
\begin{aligned}
\big\{D.T.\big\}\;
&\ge
 s(2-\alpha)\lambda\int\!\!\!\!\!\int_{Q}\theta w_x^2 dxdt
 + s^3(2-\alpha)\lambda^3\int\!\!\!\!\!\int_{Q}\theta^3 \Big(\frac{x}{a}\Big)^2w^2dxdt
\\[3pt]&-2 s^2\lambda^2e^{2}\int\!\!\!\!\!\int_{Q}\theta|\dot{\theta}|\Big(\frac{x}{a}\Big)^2w^2dxdt
-s\lambda \frac{\beta}{2}\|p\|_{L^{\infty}(0,1)} \int\!\!\!\!\!\int_{Q}\frac{|\ddot{\theta}|}{a}w^2dxdt
\\
&-se\lambda (C_1 \norm{a}_{L^{\infty}(0,1)} +c)\int\!\!\!\!\!\int_{Q}\frac{\theta}{a}w^2dxdt.
\end{aligned}
\end{equation}
Observing that there exists $C_T>0$ such that $\theta \leq C_T \theta^2,$ $|\theta\dot{\theta}|\le C_T\theta ^3$, $|\ddot{\theta}| \le C_T\theta ^2$,
one can deduce the next estimate:
\begin{equation}\label{02'}
\begin{aligned}
\big\{D.T.\big\}\;
&\ge \;
s(2-\alpha)\lambda\int\!\!\!\!\!\int_{Q}\theta w_x^2 dxdt
+ s^3(2-\alpha)\lambda^3\int\!\!\!\!\!\int_{Q}\theta^3 \Big(\frac{x}{a}\Big)^2w^2dxdt
\\[3pt]
&
-2s^2\lambda^2 e^{2}C_T\int\!\!\!\!\!\int_{Q}\theta^3\Big(\frac{x}{a}\Big)^2 w^2 dxdt
\\[3pt]&- s \lambda C_T \Big(e \big( C_1 \norm{a}_{L^{\infty}(0,1)} +c\big)+ \frac{\beta}{2}\|p~\|_{L^{\infty}(0,1)}\Big)
\int\!\!\!\!\!\int_{Q}\frac{\theta^2}{a}~w^2 dxdt.
\end{aligned}
\end{equation}
Now, we estimate
$\displaystyle -
\int\!\!\!\!\!\int_{Q}\frac{\theta^2}{a}w^2 dxdt
$\normalsize.
For $\gamma >0$ it results
\[
\begin{aligned}
\int\!\!\!\!\!\int_{Q}\frac{\theta^2}{a}w^2 dxdt
&=
\int\!\!\!\!\!\int_{Q} \left(\frac{1}{\gamma}\theta^3\left(\frac{x}{a}\right)^2w^2\right)^{\frac{1}{2}}
\left(\gamma\frac{\theta}{~x^2}\, w^2\right)^{\frac{1}{2}}dxdt
\\[3pt]&\le
\frac{1}{\gamma}\int\!\!\!\!\!\int_{Q}\theta^3\left(\frac{x}{a}\right)^2w^2 dxdt
+ \gamma\int\!\!\!\!\!\int_{Q}\frac{\theta}{~x^2}\, w^2 dxdt.
\end{aligned}\]
By the classical Hardy's inequality \cite[Lemma 5.3.1]{Davies1995} one has
\[\begin{aligned}
\int\!\!\!\!\!\int_{Q}\frac{\theta^2}{a}w^2 dxdt
& \le
\frac{1}{\gamma}\int\!\!\!\!\!\int_{Q}\theta^3\left(\frac{x}{a}\right)^2w^2 dxdt
+\gamma C_H\int\!\!\!\!\!\int_{Q}\theta w_x^2 dxdt,
\end{aligned}\]
for some positive constant $C_H.$ Substituting in \eqref{02'}, we obtain
\[
\begin{aligned}
\big\{D.T.\big\}\;
&\ge \;
s(2-\alpha)\lambda\int\!\!\!\!\!\int_{Q}\theta w_x^2 dxdt
+ s^3(2-\alpha)\lambda^3\int\!\!\!\!\!\int_{Q}\theta^3 \Big(\frac{x}{a}\Big)^2w^2dxdt
\\[3pt]
&
-2s^2\lambda^2 e^{2}C_T\int\!\!\!\!\!\int_{Q}\theta^3\Big(\frac{x}{a}\Big)^2 w^2 dxdt
\\[3pt]&- s \lambda C_T \Big(e \big( C_1 \norm{a}_{L^{\infty}(0,1)} +c\big)+ \frac{\beta}{2}\|p~\|_{L^{\infty}(0,1)}\Big)
\frac{1}{\gamma}\int\!\!\!\!\!\int_{Q}\theta^3\left(\frac{x}{a}\right)^2w^2 dxdt\\
&-s \lambda C_T \Big(e \big( C_1 \norm{a}_{L^{\infty}(0,1)} +c\big)+ \frac{\beta}{2}\|p~\|_{L^{\infty}(0,1)}\Big)\gamma C_H\int\!\!\!\!\!\int_{Q}\theta w_x^2 dxdt.
\end{aligned}
\]
By \eqref{Dis1}, \eqref{D&BT}, \eqref{BT1} and the previous inequality, we can conclude that for $s_0$ large enough and $\gamma$ small enough,
\[\begin{aligned}
&C_{\gamma}\left(s\lambda \int\!\!\!\!\!\int_{Q}\theta w_x^2dxdt
+ s^3\lambda^3\int\!\!\!\!\!\int_{Q}\theta^3 \left(\frac{x}{a}\right)^2 w^2 dxdt\right)
- se\int_0^T\theta(t)w_x^2(t,1)dt
\\[3pt]&\le
\int\!\!\!\!\!\int_{Q} g^{2}\text{\small$\frac{~e^{2s\phi}}{a}$\normalsize}~dxdt,
\end{aligned}\]
for some positive constant $C_{\gamma}$ and for all $s\ge s_0$.
\end{proof}

Recalling the definition of $w$, we have $v= e^{-s\phi}w$
and $v_{x}=  (w_{x}-s\phi_{x}w)e^{-s\phi}$. Thus, Theorem \ref{thm_Carl_bound} is proved.

\subsection{Proof of  Caccioppoli's inequality \eqref{inequality_caccio}}
This subsection is devoted to the proof of Caccioppoli's inequality \eqref{inequality_caccio}.
\begin{proof}[Proof of Lemma \ref{lemma_caccio}]
Let us consider a smooth cut-off function $\xi:[0,1] \rightarrow \mathbb{R}$ such that
$$
\left\{\begin{array}{ll}
0 \leq \xi(x) \leq 1, & \text { for all } x \in[0,1], \\
\xi(x)=1, & x \in \omega_2, \\
\xi(x)=0, & x \in(0,1) \backslash \omega_1.
\end{array}\right.
$$
Since $v$ solves \eqref{adjoint_problem}, then integrating by parts we obtain
\begin{align*}
0 &=\int_{0}^{T} \frac{d}{d t}\left( \int_{0}^{1}\left(\xi e^{s \pi}\right)^{2} v^{2} d x \right) d t \\
&=\int\!\!\!\!\!\int_{Q} \big( 2 s \pi_{t}\left(\xi e^{s \pi}\right)^{2} v^{2}+2\left(\xi e^{s \pi}\right)^{2} v (- a v_{xx} + g) \big) dx dt \\
&=2 s \int\!\!\!\!\!\int_{Q} \pi_{t}\left(\xi e^{s \pi}\right)^{2} v^{2} d x  d t - \int\!\!\!\!\!\int_{Q} ( a \xi^{2} e^{2 s \pi})_{xx}  v^2 dx dt \\
&+2 \int\!\!\!\!\!\int_{Q} a \xi^{2} e^{2 s \pi}  v_x^2 dx dt + 2 \int\!\!\!\!\!\int_{Q} \xi^{2} e^{2 s \pi} v g dx dt .
\end{align*}

Therefore,
\begin{align*}
\int\!\!\!\!\!\int_{Q} a \xi^{2} e^{2 s \pi} v_{x}^{2} dx dt &= -  s \int\!\!\!\!\!\int_{Q} \pi_{t}\left(\xi e^{s \pi}\right)^{2} v^{2} d x  d t + \frac{1}{2}\int\!\!\!\!\!\int_{Q} ( a \xi^{2} e^{2 s \pi})_{xx}  v^2 dx dt \\
&- \int\!\!\!\!\!\int_{Q} \xi^{2} e^{2 s \pi} v g d x  d t.
\end{align*}
Hence, taking into account the definition of $\xi$ and the fact that, $|\dot\theta| \leq C \theta^2$, $a \in C^1(\overline{\omega}_2)$, $\inf\limits_{x\in \omega_2} a(x) > 0$ and $\eta \in C^2(\overline{\omega}_2)$, applying Young's inequality, we infer that
\begin{align*}
&\inf _{x\in\omega_2}\{a(x)\} \int_{0}^{T}\!\!\!\int_{\omega_2} e^{2 s \pi} v_{x}^{2} dx dt
\leq \int\!\!\!\!\!\int_{Q} a \xi^{2} e^{2 s \pi} v_{x}^{2} dx dt\\
& \leq  C   \int_{0}^{T}\!\!\!\int_{\omega_1} s\theta^2 e^{2s \pi} v^{2} dx dt 
+  C \int_{0}^{T}\!\!\!\int_{\omega_1} s^2 \theta^2 e^{2s \pi} v^{2} dx dt \\
&+ C\int_{0}^{T}\!\!\!\int_{\omega_1} e^{2s \pi} v^{2} d x  d t
+ C \int_{0}^{T}\!\!\!\int_{\omega_1} e^{2 s \pi} g^{2} d x  d t \\
& \leq C \int_{0}^{T}\!\!\!\int_{\omega_1} \big( g^2 + s^2 \theta^2  v^{2}\big) e^{2s \pi} d x  d t,
\end{align*}
from which the conclusion follows.
\end{proof}
%===========================================================================================
\subsection{Compactness Theorem
}

For the proof of Theorem \ref{thm_compact_generale} below we will use the Aubin's Theorem, that we give here for the reader's convenience.
\begin{theorem}(see, e.g., \cite[Chapter 5]{an})\label{thm_Aubin}
Let $X_0$, $X_1$ and $X_2$ be three Banach spaces such that
\[
X_0 \subset X_1 \subset X_2,
\]
$X_0$, $ X_2$ are reflexives and the injection of $X_0$ into $X_1$
is compact. Let $r_0$, $r_1 \in (1, +\infty)$ and $a,b \in \R$,
$a<b$. Then the space
\[
L^{r_0}(a,b; X_0) \cap W^{1,r_1}(a,b; X_2)
\]
is compactly embedded in $L^{r_0}(a,b; X_1)$.
\end{theorem}
Hence, the next result holds.
\begin{theorem}\label{thm_compact_generale} Assume Hypothesis \ref{Hypoth_1_nondiv}. Then, the space $H^1(0, T;L^2_{\frac{1}{a}}(0,1)) \cap L^2(0, T; H^2_{\frac{1}{a}}(0,1))$
is compactly embedded in $ C([0, T]; L^2_{\frac{1}{a}}(0,1)) \cap L^2(Q).$
\end{theorem}

\begin{proof}
In the following, by $\hookrightarrow$ and $\hookrightarrow\hookrightarrow$ we shall denote continuous and compact embedding respectively.

Since
$$H^1(0, T;L^2_{\frac{1}{a}}(0,1)) \cap L^2(0, T; H^2_{\frac{1}{a}}(0,1))\hookrightarrow H^1(0, T;L^2_{\frac{1}{a}}(0,1))$$ and $$H^1(0, T;L^2_{\frac{1}{a}}(0,1))\hookrightarrow\hookrightarrow C([0, T]; L^2_{\frac{1}{a}}(0,1), $$
we have
\begin{equation}\label{embed1}
H^1(0, T;L^2_{\frac{1}{a}}(0,1)) \cap L^2(0, T; H^2_{\frac{1}{a}}(0,1))\hookrightarrow
\hookrightarrow C([0, T]; L^2_{\frac{1}{a}}(0,1)).
\end{equation}
On the other hand, by observing that $H^2_{\frac{1}{a}}(0,1) \hookrightarrow H^1_{\frac{1}{a}}(0,1)$ and $L^2_{\frac{1}{a}}(0,1) \hookrightarrow L^2(0,1) \hookrightarrow H^{-1}(0,1)$, we obtain
$$H^1(0, T;L^2_{\frac{1}{a}}(0,1)) \cap L^2(0, T; H^2_{\frac{1}{a}}(0,1))\hookrightarrow H^1(0, T;H^{-1}(0,1)) \cap L^2(0, T; H^1_{\frac{1}{a}}(0,1)).$$
Next, taking $r_0= r_1=2$, $X_0 = H^1_{\frac{1}{a}}(0,1)=H^1_{0}(0,1)$ (by Proposition \ref{Prop_Hardy_nondiv}), $X_1 =
L^2(0,1)$, $X_2 = H^{-1}(0,1)$, $a=0$ and $b=T$, by Theorem \ref{thm_Aubin}, one has
\begin{equation}\label{embed3}
H^1(0, T;H^{-1}(0,1)) \cap L^2(0, T; H^1_{\frac{1}{a}}(0,1)) \hookrightarrow\hookrightarrow L^2(Q).
\end{equation}
Finally, from \eqref{embed1} and \eqref{embed3}, we conclude that
\begin{equation*}
H^1(0, T;L^2_{\frac{1}{a}}(0,1)) \cap L^2(0, T; H^2_{\frac{1}{a}}(0,1)) \hookrightarrow\hookrightarrow C([0, T]; L^2_{\frac{1}{a}}(0,1)) \cap L^2(Q).
\end{equation*}
\end{proof}
%=================================
%=================================
%=================================


\begin{thebibliography}{99}
\bibitem{Hajjaj2013}
E. M. Ait Ben Hassi, F. Ammar Khodja, A. Hajjaj and L. Maniar, {\it Carleman estimates and null controllability of coupled degenerate systems}, Evol. Equ. Control Theory, 2 (2013), 441–459.
%======================================
\bibitem{Alabau2006}
F. Alabau-Boussouira, P. Cannarsa, G. Fragnelli, {\it Carleman estimates for degenerate parabolic operators with application to null controllability},  J. Evol. Equ. 6 (2006), 161-204.
%======================================

\bibitem{AHMS19} B. Allal, A. Hajjaj, L. Maniar, J. Salhi,
{\it Lipschitz stability for some coupled degenerate parabolic systems with locally distributed observations of one component},  Mathematical Control \& Related Fields, \textbf{10} (2020), 643--667.
%======================================

\bibitem{Allal2020} 
B. Allal, G. Fragnelli,
{\it Null controllability of degenerate parabolic equation with memory}, Math. Methods Appl. Sci., in press. https://doi.org/10.1002/mma.7342.
%======================================
\bibitem{AFS2020} 
B. Allal, G. Fragnelli, J. Salhi,
{\it Null controllability for a singular heat equation with
a memory term}, Electron. J. Qual. Theory Differ. Equ., \textbf{14} (2021), 1--24.
%%======================================


\bibitem{an} A. Ani\c{t}a,
{\it Analysis and Control of Age-Dependent Population Dynamics},
Mathematical Modelling: Theory and Applications, 11. Kluwer
Academic Publishers, Dordrecht, 2000.
%======================================
 \bibitem{bfr} V. Barbu, A. Favini, S. Romanelli,
   {\it Degenerate evolution equations and regularity of their
   associated semigroups}, Funkcial. Ekvac \textit{39} (1996), 421-448.
%======================================


\bibitem{Biccari2019} U. Biccari and V. Hern\'{a}ndez-Santamar\'{\i}a,
{\it Null controllability of linear and semilinear nonlocal heat equations with an additive integral kernel},
SIAM J. Control Optim. 57, 4 (2019), 2924-2938.
%======================================

\bibitem{Calsina} A. Calsina, C. Perello, J. Saldana,  {\it Non-local reaction-diffusion equations modelling predator-prey coevolution}, Publ.
Mat. \textbf{38} (1994), 315-325.
%======================================
\bibitem{Campiti1998} M. Campiti, G. Metafune and D. Pallara,
{\it Degenerate self-adjoint evolution equations on the unit interval}, Semigroup Forum, 57 (1998), 1-36.
%=====================================
\bibitem{CFR2008} P. Cannarsa, G. Fragnelli, and D. Rocchetti, {\it Controllability results for a class of onedimensional degenerate parabolic problems in nondivergence form}, J. Evol. Equ. 8 (2008),
no. 4, 583-616.
%=====================================
\bibitem{chipot2000} M. Chipot, {\it Elements of Nonlinear analysis}, Birkh\"{a}user Advanced Text, 2000.
%=====================================
\bibitem{Davies1995}
E. B. Davies, {\it Spectral theory and differential operators}, Cambridge Studies in Advanced
Mathematics \textbf{42}, Cambridge University Press, Cambridge, 1995.
%=====================================
%======================================
\bibitem{Evans}
L. C. Evans,  {\it Partial Differential Equations}, second edition,  Graduate Studies in Mathematics 19, 2010.
%%=====================================
%=====================================
\bibitem{Favini}
A. Favini an. Yagi,  {\it Degenerate Differential Equations in Banach Spaces}, Pure anpplied Mathematics: A Series of Monographs and Textbooks 215, M.Dekker, New York, 1998.
%%=====================================
%=====================================
\bibitem{cara2004} E. Fern\'{a}ndez-Cara, S. Guerrero, O. Y. Imanuvilov, and J.-P. Puel,
{\it Local exact controllability of the Navier-Stokes system},
J. Math. Pures Appl., \textbf{83} (2004), 1501-1542.
%=====================================
%=====================================
\bibitem{cara2006} E. Fern\'{a}ndez-Cara, S. Guerrero,
{\it Global Carleman estimates for solutions of parabolic systems defined by transposition and some applications to controllability},
Applied Mathematics Research eXpress, Volume 2006, Article ID 75090, Pages 1-31.
%=====================================

\bibitem{cara} E. Fern\'{a}ndez-Cara and S. Guerrero,
{\it Global Carleman inequalities for parabolic systems and applications to null controllability},
SIAM J. Control Optim., \textbf{45} (2006), 1395--1446.
%=====================================
\bibitem{CLZ2016} E. Fern\'andez-Cara, Q. L\"{u}, and E. Zuazua, {\it Null controllability of linear heat and wave equations with nonlocal spatial terms}, SIAM J. Control Optim., \textbf{54} (2016), pp. 2009-2019.
%=====================================
\bibitem{FMpress} G. Fragnelli, D. Mugnai, Control of degenerate and singular parabolic equation, BCAM
SpringerBriefs, ISBN 978-3-030-69348-0.
%======================================
\bibitem{yamamoto2020} G. Fragnelli, M. Yamamoto, {\it Carleman estimates and controllability for a degenerate structured population model}, Appl. Math. Optim, DOI: 10.1007/s00245-020-09669-0.
%=====================================

%%=====================================
\bibitem{FI1996}
A. V. Fursikov and O. Y. Imanuvilov,
{\it Controllability of evolution equations},
Lect. Notes Ser. \textbf{34}, Seoul National University,
Seoul, 1996.
%======================================
\bibitem{Furter} J. Furter, M. Grinfeld , {\it Local vs. non-local interactions in population dynamics}, J. Math. Biol. \textbf{27} (1989), 65-80.
%======================================
%=====================================

\bibitem{gueye} M. Gueye,
{\it Insensitizing controls for the Navier-Stokes equations},
Ann. Inst. H. Poincar\'e Anal. Non Lin\'eaire, \textbf{30} (2013), 825-844.
%=====================================

\bibitem{Lissy2018} P. Lissy and E. Zuazua, {\it Internal controllability for parabolic systems involving analytic nonlocal terms}, Chin. Ann. Math. Ser. B, \textbf{39} (2018), pp. 281-296.
%======================================
\bibitem{Lions71}
J.L. Lions,  {\it Optimal control of systems governed by partial differential equations}, Springer-Verlag, 1971.
%======================================
\bibitem{Lorenzi11} A. Lorenzi, {\it Two severely ill-posed linear parabolic problems}, in AIP Conference Proceedings,
AIP Conf. Proc. 1329, AIP, Melville, NY, 2011, pp. 150-169.
%======================================
\bibitem{Martinez06}
P. Martinez and J. Vancostenoble, {\it Carleman estimates for one-dimensional degenerate heat equations}, J. Evol. Equ. 6 (2006), no. 2, 325–362, DOI 10.1007/s00028-006-0214-6.
%======================================
\bibitem{Micu} S. Micu and T. Takahashi, {\it Local controllability to stationary trajectories of a Burgers equation with nonlocal viscosity}, J. Differential Equations, \textbf{264} (2018), pp. 3664-3703.
%======================================
\bibitem{simon}
J. Simon, {\it Compact sets in the space $L^p(0,T;B)$},
Ann. Mat. Pura Appl., \textbf{146} (1986), 65--96
%======================================
\bibitem{TG2016}
Q. Tao and H. Gao,
{\it On the null controllability of heat equation with memory}, J. Math. Anal. Appl., \textbf{440} (2016) 1--13.
%%=====================================
\bibitem{Zheng} Y. Zheng and Z. B. Fang, {\it Qualitative properties for a pseudo-parabolic equation with nonlocal reaction term}, Boundary
Value Problems, vol. 2019, no. 1, p. 134, 2019.
%\bibitem{Montoya2018}
%C. Montoya and L. de Teresa, Robust Stackelberg controllability for the Navier-Stokes equations, NoDEA Nonlinear Differential Equations Appl., 25 (2018), 46, https://doi.org/10.
%1007/s00030-018-0537-3.
\end{thebibliography}
\end{document}